\documentclass[12pt]{amsart}

\usepackage{amsmath}
\usepackage{amsthm}
\usepackage{amssymb}
\usepackage[left=3cm,right=3cm,top=3cm,bottom=3cm]{geometry}
\usepackage{amscd}
\usepackage{stmaryrd}
\usepackage[normalem]{ulem}
\usepackage{MnSymbol}
\usepackage{bbm}
\usepackage[arrow, curve, matrix]{xy}
\usepackage{tikz-cd}
\usepackage{accents}

\usepackage[utf8]{inputenc}
\usepackage[cyr]{aeguill}
\usepackage{xspace}

\usepackage{tabu, multirow}

\usepackage{listings}
\usepackage{verbatim}

\usepackage{mathbbol}
\usepackage{amssymb}             

\DeclareSymbolFontAlphabet{\mathbb}{AMSb}%
\DeclareSymbolFontAlphabet{\mathbbl}{bbold}

\usepackage{enumerate}

\usepackage{array}

\usepackage{hyperref}
\usepackage{varioref}
\usepackage[hyperpageref]{backref}

\usepackage{xcolor}
\hypersetup{
    colorlinks,
    linkcolor={red!60!black},
    citecolor={blue!60!black},
    urlcolor={blue!50!black}
}

\theoremstyle{thm} \newtheorem{thm}{Theorem}
\theoremstyle{thm} \newtheorem*{thm*}{Theorem}
\newtheorem{prop}{Proposition} [section]
\newtheorem{lem}[prop]{Lemma}
\newtheorem{corol}[prop]{Corollary}
\newtheorem{conj}[prop]{Conjecture}
\theoremstyle{definition} 
\theoremstyle{remark} 
\theoremstyle{remark} \newtheorem{rem}[prop]{Remark}
\theoremstyle{definition} \newtheorem{defi}[prop]{Definition}
\theoremstyle{definition} \newtheorem*{nota}{Notation}

\newcommand{\quotientd}[2]{{\left.\raisebox{.2em}{$#1$}\middle\slash\raisebox{-.2em}{$#2$}\right.}}

\newcommand{\norm}[1] {\| #1 \| }

\makeatletter
\newcommand{\xleftrightarrow}[2][]{\ext@arrow 3359\leftrightarrowfill@{#1}{#2}}
\newcommand{\xdashrightarrow}[2][]{\ext@arrow 0359\rightarrowfill@@{#1}{#2}}
\newcommand{\xdashleftarrow}[2][]{\ext@arrow 3095\leftarrowfill@@{#1}{#2}}
\newcommand{\xdashleftrightarrow}[2][]{\ext@arrow 3359\leftrightarrowfill@@{#1}{#2}}
\def\rightarrowfill@@{\arrowfill@@\relax\relbar\rightarrow}
\def\leftarrowfill@@{\arrowfill@@\leftarrow\relbar\relax}
\def\leftrightarrowfill@@{\arrowfill@@\leftarrow\relbar\rightarrow}
\def\arrowfill@@#1#2#3#4{%
  $\m@th\thickmuskip0mu\medmuskip\thickmuskip\thinmuskip\thickmuskip
   \relax#4#1
   \xleaders\hbox{$#4#2$}\hfill
   #3$%
}
\makeatother

\makeatletter
\def\blfootnote{\xdef\@thefnmark{}\@footnotetext}
\makeatother

\makeatletter
\newcommand{\thickhline}{%
    \noalign {\ifnum 0=`}\fi \hrule height 1pt
    \futurelet \reserved@a \@xhline
}
\newcolumntype{"}{@{\hskip\tabcolsep\vrule width 1pt\hskip\tabcolsep}}
\makeatother

\title{Generalized algebraic Morse inequalities and jet differentials}

\author{Beno\^it Cadorel}
\address{}
\email{benoit.cadorel@univ-lorraine.fr}
\address{Institut \'Elie Cartan de Lorraine \\ UMR 7502 \\ Universit\'e de Lorraine, Site de Nancy \\ B.P. 70239, F-54506 Vandoeuvre-l\`es-Nancy Cedex}

\begin{document}

\begin{abstract}We give a fully algebraic proof of an important theorem of Demailly, stating the existence of many Green-Griffiths jet differentials on a complex projective manifold of general type. To this end, we introduce a new algebraic version of the Morse inequalities, which we use in our proof as an algebraic counterpart to Demailly's and Bonavero's holomorphic Morse inequalities. \end{abstract}

\maketitle

\addcontentsline{toc}{section}{Introduction}

Jet differentials are one of the fundamental tools that were developed in the last decades to study complex hyperbolicity, and in particular the Green-Griffiths-Lang conjecture:

\begin{conj}[Green-Griffiths \cite{GG80}, Lang \cite{lang87}] \label{conjGGL} Let $X$ be a complex projective manifold of general type. Then, there exists a proper algebraic subset $\mathrm{Exc}(X) \subsetneq X$, such that $X$ is \emph{Brody hyperbolic modulo $\mathrm{Exc}(X)$} i.e. for any non constant holomorphic map $f : \mathbb C \longrightarrow X$, we have $f(\mathbb C) \subseteq \mathrm{Exc}(X)$.
\end{conj}

Recall that we say that a manifold $X$ is of general type if its canonical line bundle $K_X$ is big, i.e. if we have the maximal growth $h^0(X, K_X^{\otimes m}) \geq C \; m^{\dim X}$ for some $C > 0$. Non constant holomorphic maps $\mathbb C \longrightarrow X$ are usually called \emph{entire curves} on $X$.
\medskip

The relevance of jet differential techniques for the study of complex hyperbolicity problems has first been evidenced by Green and Griffiths \cite{GG80}. For any integers $k, m \geq 1$, they have constructed a vector bundle $E_{k,m}^{GG} \Omega_X \longrightarrow X$, the so-called \emph{bundle of Green-Griffiths jet differentials}, whose local sections represent \emph{holomorphic ordinary differential equations of order $k$ and degree $m$} on $X$ (acting on germs of holomorphic curves). The main interest of these jet differentials in view of Conjecture \ref{conjGGL} comes from the following fundamental theorem. 

\begin{thm*} [\cite{SY96, dem97}] \label{thmannfund} Let $A$ be an ample line bundle on $X$. Let $k, m \in \mathbb N^\ast$, and let $P \in H^0(X, E_{k, m}^{GG} \Omega_X \otimes \mathcal O(- A))$. Then, any non constant holomorphic map $f : \mathbb C \longrightarrow X$ is a solution to the differential equation $P$, i.e.  $P(f; f', ..., f^{(k)}) \equiv 0$. \end{thm*}

The theorem above is the first step of a general strategy to find restrictions on the geometry of entire curves, in the hope of eventually proving the Green-Griffiths-Lang conjecture. In many contexts, this strategy has permitted to obtain strong hyperbolicity results e.g. for hypersurfaces of $\mathbb P^{n+1}$ (see e.g. \cite{DMR10, ber15, bro17, siu15, deng16, dem18, BK19}, complements of hypersurfaces of $\mathbb P^n$ (see \cite{Dar15, BD19}), surfaces of general type (see \cite{bog77, McQ98, RR12})... The field of application of these jet differential techniques has also been recently extended to the orbifold setting by Campana, Darondeau, and Rousseau in \cite{CDR18}.
\medskip

The strategy sketched above will work if we are able to prove the existence of many global jet differential equations on a given manifold of general type. The most general result in this direction is due to Demailly \cite{dem11}, and can be stated as follows.

\begin{thm} \label{thmprinc} Let $X$ be a complex projective manifold of general type. Then, for a fixed $k \in \mathbb N$ large enough, the Green-Griffiths sheaf of algebras $E_{k, \bullet}^{GG} \Omega_X$ is \emph{big}, i.e. there is maximal growth $h^0(X, E_{k, m}^{GG} \Omega_X) \geq C m^{n + nk - 1}$ with $C > 0$, if $m \gg 1$ is divisible enough.

In particular, if $A$ is an ample line bundle on $X$, and if $m \gg k \gg 1$, we have
$$
H^0(X, E_{k, m}^{GG} \Omega_X \otimes \mathcal O(- A)) \neq 0.
$$
\end{thm}

The proof of Theorem \ref{thmprinc} given by Demailly in \cite{dem11} is fundamentally analytic in nature: it is based on the \emph{holomorphic Morse inequalities} he introduced in \cite{dem85}, and that were later extended to the singular setting by Bonavero \cite{bonavero98}. 
\medskip

Our main goal in these notes is to give a fully algebraic and, we hope, essentially self-contained proof of Theorem \ref{thmprinc}. The general strategy will be to exhibit algebraic counterparts to the concepts introduced in \cite{dem11}. In particular, we introduce a new algebraic version of the holomorphic Morse inequalities.

\subsection{Algebraic Morse inequalities} The current main candidate for an algebraic version of the holomorphic Morse inequalities are the \emph{algebraic Morse inequalities} of Demailly \cite{dem96} and Angelini \cite{ang96}. In their simplest form, due to Siu \cite{siu93}, they can be stated as follows. Let $X$ be a complex projective manifold of dimension $n$, and let $L$ be a line bundle on $X$. Assume that $L = \mathcal O(A - B)$, where $A, B$ are nef divisors on $X$. Then, for $m \gg 1$, we have
$$
h^0(X, L^{\otimes m}) - h^1(X, L^{\otimes m}) \geq \frac{m^n}{n!} \left(A^n - n B \cdot A^{n-1} \right) + O(m^{n-1}).
$$
In particular, if $A^n > n B \cdot A^{n-1}$, then $h^0(X, L^{\otimes m})$ has maximal growth, and $L$ is big.
\medskip

A first natural idea to try to prove Theorem \ref{thmprinc} with algebraic methods, is to apply these holomorphic Morse inequalities to the tautological line bundle of the \emph{Green-Griffiths jet spaces}. If $X$ is a complex manifold, these jet spaces are projective fiber bundles $X_k^{GG} \overset{\pi_k}{\longrightarrow} X$ ($k \in \mathbb N^\ast$), each one endowed with an (orbifold) tautological line bundle $\mathcal O_k^{GG}(1)$, such that $(\pi_k)_\ast \mathcal O_k^{GG}(m) = E_{k, m}^{GG} \Omega_X$. Showing that $E_{k, \bullet}^{GG} \Omega_X$ is big for some $k$, amounts to showing that $\mathcal O_k^{GG}(m)$ is big; in this situation, we can then try to apply the previous algebraic Morse inequalities. To do this, we need to write $\mathcal O_k^{GG}(m) = \mathcal O(A - B)$, where $A, B$ are nef divisors on $X_k^{GG}$, and then to compute $A^N - N A^{N-1} \cdot B$, where $N = \dim X_k^{GG}$. A natural way to proceed is to remark that $\mathcal O_k^{GG}(m)$ is relatively ample, and to choose $B = \pi_k^\ast H$, where $H$ is an ample divisor on $X$, sufficiently positive so that $\mathcal O (A) = \mathcal O_k^{GG}(m) \otimes \mathcal \pi_k^\ast \mathcal O(H)$ is itself nef.
\medskip

This general strategy has been followed by Diverio to the case where $X$ is a hypersurface of $\mathbb P^{n+1}$ of degree $d$ in \cite{div08, div09} (more precisely, Diverio deals with the \emph{Demailly-Semple jet tower} rather than with the spaces $X_k^{GG}$).  In this situation, he shows the existence of many jet differentials of order $k \geq n$, as soon as $d \geq d(n)$, for some constant $d(n) \in \mathbb N$. Unfortunately, this algebraic method does not seem to give the bigness of $E_{k, \bullet}^{GG} \Omega_X$ when $d \geq n + 3$ and $k \gg 1$, which would be the expected bound on $d$ according to Theorem \ref{thmprinc}. There seems here to be a discrepancy between the results the algebraic and analytic methods can provide. \medskip 

It seems to us this discrepancy comes from a too restrictive setting for the statement of the algebraic Morse inequalities. Rather than dealing only with the case where $L = A - B$, with $A, B$ nef, it would be much more flexible to be able to deal with any difference of \emph{effective} divisors $A, B$. 

In the following discussion, we propose to follow this idea, and to give accordingly a new algebraic version of the Morse inequalities. We are quickly led to stating Morse inequalities in terms of \emph{stratifications} on our varieties, as follows. Let $L \longrightarrow X$ be a line bundle over a complex manifold, and let $e$ be a trivialization of $L$ over some affine open subset $U \subseteq X$, where $X \setminus U$ is the support of a Cartier divisor $D$ such that $L= \mathcal O(D)$. We can now extract a particular combinatorial data of this situation: 1) write $D = D^+ - D^-$, where $D^+$, $D^-$ are effective ; 2) note the multiplicities of $e$ along the irreducible components of $D^+$ and $D^-$ ; 3) restrict $L$ to $D$, and find a trivialization of this restriction on a Zariski dense open subset of $D$ ; 4) repeat this operation with $X$ replaced by $D$. Inductively, this defines a stratification $\Sigma$ on $X$, with the data of a trivialization $\mathbf e$ over $\Sigma$ (see Section \ref{sectstrat} for the proper definitions). Then the data of $\underline{\Sigma} = (\Sigma, \mathbf e)$ and of the multiplicities computed along the way permits to define \emph{truncated Chern intersection numbers} $\deg c_1(X, \underline{\Sigma})^{n}_{[\leq i]}$, for all $0 \leq i \leq n$. 

We can also extend these definitions to the case of $\mathbb Q$-line bundles, to formulate generalized Morse inequalities as follows.

\begin{thm} \label{thmintromorse} Let $X$ be a normal projective variety of dimension $n$. Let $L$ be a $\mathbb Q$-line bundle on $X$, and let $\underline{\Sigma}$ be a trivialized stratification adapted to $L$. Let $M$ be another line bundle on $X$. Then, for each integer $0 \leq i \leq n$, and any $m$ divisible enough, we have
\begin{enumerate}[(i)]
\item (Strong Morse inequalities) 
$$
\sum_{0 \leq j \leq i}  (-1)^{j + i} h^j (X, M \otimes L^{\otimes m}) \leq (-1)^i \left( \deg c_1(L, \underline{\Sigma})^n_{[\leq i]} \right) \frac{m^n}{n!} + O(m^{n-1})
$$

\item (Weak Morse inequalities) 
$$
h^i (X, M \otimes L^{\otimes m}) \leq (-1)^i \left( \deg c_1(L, \underline{\Sigma})^n_{[i]} \right) \frac{m^n}{n!} + O(m^{n-1})
$$

\item (Asymptotic Riemann-Roch formula)
$$
\chi(X, M \otimes L^{\otimes m}) = \left( \deg c_1(L, \underline{\Sigma})^n_{[\leq n]} \right) \frac{m^n}{n!} + O(m^{n-1})
$$
\end{enumerate}
\end{thm}

The core of our proof is very similar in spirit to the one of Angelini \cite{ang96}: it is an induction on $\dim X$, using quite standard dimensional considerations on long exact sequences of coherent sheaves. However, some difficulties will appear due to the need to work with the singular varieties appearing in our stratifications, as well as ramified coverings used to deal with the $\mathbb Q$-line bundles.

\subsection{Weighted projectivized bundles and jet spaces} To prove Theorem \ref{thmprinc}, a natural idea would be to proceed as in \cite{dem11}, and to apply the Morse inequalities to the line bundles $\mathcal O_k^{GG}(1) \longrightarrow X_k^{GG}$. In these notes, we will follow a rather different strategy, which is based on two simplifying ideas (which seem also relevant in the analytic context). 

First of all, following a remark we made in \cite{cad17}, we can reduce the study of the jet spaces $X_k^{GG}$ to the one of a \emph{weighted projectivized bundle}, by using a construction implicitly used in \cite{dem11}. Its geometric interpretation is the existence of a deformation of $X_k^{GG}$ into the weighted projectivized bundle $P_k = \mathbb{P}_X(\Omega_X^{(1)} \oplus ... \oplus \Omega_X^{(k)}),$ (see Section \ref{sectweightproj} for a proper definition), and of an (orbifold) line bundle over this family of deformations, restricting to the tautological line bundles $\mathcal O_k^{GG}(1) \longrightarrow X_k^{GG}$ and $\mathcal O_k(1) \longrightarrow P_k$. In this situation, the semi-continuity properties of $h^0 - h^1$ (see \cite{dem95}) yields
\begin{equation} \label{eqineqfund}
h^0 ( X_k^{GG}, \mathcal O_k^{GG}(m)) - h^1 ( X_k^{GG}, \mathcal O_k^{GG}(m)) \geq  h^0(P_k^{GG}, \mathcal O_k(m)) - h^1(P_k^{GG}, \mathcal O_k(m)).
\end{equation}

This inequality can also be shown using a very simple argument on filtered algebras which was communicated to me by L. Darondeau. To show that the left hand side is large, we just have to bound from below the right hand side, which is in fact the $h^0 - h^1$ of a natural graded algebra on $E_{k,m}^{GG} \Omega_X$: this right hand side is equal to  
\begin{equation} \label{eqchi1sym}
(h^0 - h^1) \left( \sum_{l_1 + 2 l_2 + ... + k l_k = m} S^{l_1} \Omega_X \otimes ... \otimes S^{l_r} \Omega_X \right), 
\end{equation}
and we get to our first reduction step, which states that it is enough to bound from below the quantity \eqref{eqchi1sym}, which is completely described \emph{only in terms of the cotangent bundle $\Omega_X$}.
\medskip
 
As another important reduction step, we can use a version of the \emph{splitting principle} implying that to get an asymptotic lower bound for \eqref{eqchi1sym}, it is actually enough to deal with the case where $\Omega_X = L_1 \oplus ... \oplus L_n$ is a direct sum of line bundles. We can then expend \eqref{eqchi1sym} in terms of direct sums of various products of powers of the $L_i$, and apply the Morse inequalities to each one of the direct factors so obtained. This computation eventually makes appear an integral over the standard $(kn-1)$-dimensional simplex $\Delta^{kn-1}$; we can actually state the following general result (see Section \ref{sectstatement} and Theorem \ref{thmineqintegral} for more precise and general statements). 

\begin{thm} \label{thmintroint} Let $X$ be a complex projective manifold of dimension $n$, and let $L_1, ..., L_r$ be line bundles on $X$. Let $a_1, ..., a_r \in \mathbb N_{\geq 1}$. Fix a stratification $\Sigma$ on $X$, and let $\mathbf{e}_1, ..., \mathbf{e}_r$ be trivializations of the $L_i$ on $\Sigma$. Then, for all $i$ $(0 \leq i \leq n)$, there exists a piecewise polynomial function $\upsilon_{[\leq i]} :\Delta^{r-1} \longrightarrow \mathbb R$, constructed explicitly in terms of $(\Sigma, \mathbf{e}_1, ..., \mathbf{e}_r)$, such that for all $m \in \mathbb N$:
\begin{align} \nonumber
\chi^{[i]} (X, \bigoplus_{a_1 l_1 + ... + a_r l_r = m} & L_1^{\otimes l_1} \otimes ...  \otimes L_r^{\otimes l_r} )  \\ \label{eqintrointegral}
& \leq  \frac{\mathrm{gcd}(a_1, ..., a_r)}{a_1 ... a_r} \binom{n + r - 1}{r - 1} \left[ \int_{\Delta^{r-1}} \upsilon_{[\leq i]} dP \right] \frac{m^{n+ r - 1}}{(n+ r - 1)!} + o(m^{n+r - 1}),
\end{align}
where $dP$ is the invariant probability measure on the simplex $\Delta^{r-1}$.
\end{thm}

Suppose now that we are given a trivialized stratification $\underline{\Sigma} = (\Sigma, \mathbf{e})$, adapted to $K_X = L_1 \otimes ... \otimes L_n$. We can then construct trivializations $\mathbf{e}_i$ of the $L_i$ on $\Sigma$ (after possibly refining $\Sigma$), whose tensor product gives back $\mathbf{e}$. The next part of the work is to estimate the integral in \eqref{eqintrointegral} as we apply this inequality to \eqref{eqchi1sym}, letting $k \longrightarrow + \infty$. This estimation is very close to the computations done in \cite{dem11}: we can actually present them in a probabilistic manner, further elaborating on Demailly's Monte-Carlo method. As in the analytic situation, we observe an "averaging" phenomenon: as $k \longrightarrow + \infty$, the integral of \eqref{eqintrointegral} gets closer to its mean value over the simplex $\Delta^{r-1}$; a simple computation shows that this mean value is proportional to $\deg c_1(K_X, \underline{\Sigma})^n_{[\leq i]}$.

Letting $i = 1$, this implies that Theorem \ref{thmprinc} will be proved if we can find a stratification for which $\deg c_1(K_X, \underline{\Sigma})^n_{[\leq 1]} > 0$, whenever $K_X$ is big. As often when passing analytic arguments to an algebraic context, this stratification will actually exist only on some ramified covering $X' \overset{p}{\longrightarrow} X$, produced using Kawamata and Bloch-Gieseker lemmas. Also, $\Sigma$ will actually be adapted to an ample subsheaf of $\mathcal O(p^\ast K_X)$ rather that to $K_X$ itself.

Putting everything together, we get the following result, which implies immediately Theorem \ref{thmprinc}.

\begin{thm} \label{thmfulldetail} Let $X$ be a projective manifold of general type, of dimension $n$. Then, for all small $\epsilon > 0$, there exists 
\begin{enumerate} \itemsep=0em
		\item a generically finite, proper morphism $p: X' \longrightarrow X$;
		\item a decomposition $p^\ast K_X = A + E$ into ample and effective divisors;
		\item a trivialized stratification $\underline{\Sigma}$ on $X'$, adapted to $A$, such that $$\deg c_1(A, \underline{\Sigma})^n_{[\leq 1]} > (\deg p) (\mathrm{vol}(K_X) - \epsilon) > 0.$$
\end{enumerate}

	Moreover, when $m \gg k \gg 1$, and $m$ is divisible enough, we have
$$
h^0(X', p^\ast E_{k,m}^{GG} \Omega_X) \geq \frac{(\log k)^n}{n! (k!)^n} \left( \deg c_1(A, \underline{\Sigma})^n_{[\leq 1]} - O(\frac{1}{ \log k}) \right) \frac{m^{n + nk - 1}}{(n+kr - 1)!}  + o(m^{n + kr - 1}).
$$
This implies that for $k \gg 1$, $p^\ast E_{k, \bullet}^{GG} \Omega_X$, and hence $E_{k, \bullet}^{GG} \Omega_X$, is big.
\end{thm}

\subsection{Organization of the paper} These notes will essentially be divided in four parts. In the first three of them, we prove our general results about Morse inequalities, and we detail our reduction steps for the proof of Theorem \ref{thmprinc}. We will finally prove the main theorem in the last part. 

\begin{enumerate}
\item Section 1:  The basic definitions of trivialized stratifications and truncated intersection numbers will be given in Sections \ref{sectstrat} and \ref{sectruncatedchern}. We will then prove Theorem \ref{thmintromorse} in Section \ref{sectmorseineq}.

\item Section 2: We will present the main setup of the proof of Theorem \ref{thmprinc}, as a motivation for the more general results proved in Section 3.
\begin{enumerate} \item In Section \ref{sectweightproj}, we recall the basic definitions of Green-Griffiths jet differentials, and of direct sums of vector bundles. We then give our first reduction step, which shows that it suffices to apply the Morse inequalities to the symmetric powers of a weighted direct sum of the form $\Omega_X^{(1)} \oplus ... \oplus \Omega_X^{(k)}$.
		\item In Section \ref{sectrel}, we give a relative version of the results of Section \ref{sectweightproj}, which also shows that we can study twisted direct sums of the form $\mathcal O(-E) \otimes (\Omega_X^{(1)} \oplus ... \oplus \Omega_X^{(k)})$ for some effective divisor $E$. It will be important in the proof of Theorem \ref{thmprinc} to be able to deal with the case where $K_X$ is merely big, and not necessarily ample.
\item Section \ref{sectredsum} is devoted to the version of the splitting principle mentioned above: this is our second reduction step. 
\end{enumerate}

\item Section 3: we prove Theorem \ref{thmintroint}, and we give a variant where we introduce twists by an auxiliary $\mathbb Q$-line bundle.
\item Section 4: we put the results of Section 3 and 4 together: we explain how to construct adapted stratifications in the case where $K_X$ is big, and we complete the proof of Theorem \ref{thmprinc}.
\item Annex: we gathered here some useful computations related to the integration of linear forms over simplexes of $\mathbb R^m$, and associated probability estimates.
\end{enumerate}

\subsection*{Acknowledgments.}

I would like to thank Damian Brotbek, Fr\'ed\'eric Campana, Julien Grivaux, Henri Guenancia, Gianluca Pacienza and Erwan Rousseau for our enriching discussions around this subject. I am also grateful to Jean-Pierre Demailly for several useful explanations about his original article, to which I am of course very indebted. Finally, I address special thanks to Lionel Darondeau for many enlightening and motivating discussions all along the conception of these notes.
\medskip

During the preparation of this work, the author was partially supported by the ANR Programme: D\'efi de tous les savoirs (DS10) 2015, "GRACK", Project ID : ANR-15-CE40-0003.

\section{Truncated Chern classes}

\subsection{Stratifications} \label{sectstrat}

In the following, $X$ will denote a normal complex projective variety of dimension $n$. For us, a \emph{variety} will be an irreducible and reduced complex scheme. 
\medskip

The first concept we would like to define is a basic notion of \emph{stratification} on $X$. It will be constructed inductively, using the following elementary step. 

\begin{defi} \label{defistratum} Let $Y_n$ (resp. $Y_{n-1}$) be a normal reduced scheme with irreducible connected components, of pure dimension $n$ (resp. $n-1$). A morphism $Y_{n-1} \overset{f}{\longrightarrow} Y_n$ is called a \emph{stratum} of $Y_n$ if it factors through a reduced Weil divisor $D_n$ on $Y_n$ such that:
\begin{enumerate}[(a)]
\item $U_n = X_n \setminus \mathrm{Supp}(D_n)$ is an open dense affine subset of $X$;
\item \label{itembij} $f : Y_{n-1} \longrightarrow D_n$ is a proper birational morphism. In particular, $f$ associates pairwise the irreducible components of $Y_{n-1}$ and of $D_n$.
\end{enumerate}
 \end{defi}

We can now define a stratification as follows.

\begin{defi} \label{defistratification} A \emph{stratification} of $X$ is a sequence of strata $X_0 \longrightarrow X_1 \longrightarrow ... \longrightarrow X_n = X$.
\end{defi}

We now define a notion of compatibility between a stratification $\Sigma$ and a vector bundle $E$, holding when there exists a trivialization of $E$ over $\Sigma$, in the following sense.

\begin{defi} Let $E \longrightarrow X$ be a vector bundle of rank $r$, and let $\Sigma \equiv X_0 \overset{f_0}{\longrightarrow} X_1 \overset{f_1}{\longrightarrow}  ... \overset{f_{n-1}}{\longrightarrow} X_n = X$ be a stratification of $X$.

A \emph{trivialization} $\mathbf e$ of $E$ over $\Sigma$, is the data, for each $i$ ($0 \leq i \leq n$), of a trivialization $(e^i_j)_{1 \leq j \leq r}$ of $q_i^\ast E$ on $U_i$, where $U_i = X_i \setminus f_{i-1}(X_{i-1})$, and $q_i : X_i \longrightarrow X$ is the natural map. A \emph{trivialized stratification} $\underline{\Sigma} = (\Sigma, \mathbf e)$ for $E$ is the data of a stratification $\Sigma$ of $X$ and of a trivialization $\mathbf e$ of $E$ over $\Sigma$.

We say that $\Sigma$ is \emph{adapted} to $E$ if there exists a trivialization of $E$ over $\Sigma$.
\end{defi}

To deal with $\mathbb Q$-line bundles, it will be useful to extend slightly the definition above.

\begin{defi}
	Let $L \longrightarrow X$ be a $\mathbb Q$-line bundle, i.e. a formal root $L = N^{\otimes 1/d}$, where $N$ is a line bundle on $X$, and $d \in \mathbb N_{\geq 1}$. Let $(\Sigma, \mathbf{e})$ be a trivialized stratification for $N$. We say will say that the formal data $\frac{1}{d} \mathbf{e}$ is a \emph{fractional trivialization} of $L$ on $\Sigma$; we will simply call the data $\underline{\Sigma} = (\Sigma, \frac{1}{d} \mathbf{e})$ a \emph{trivialized stratification} for $L$.
\end{defi}

Note that the above definition also make sense if $N$ is a standard line bundle. The proof of the following proposition is straightforward.

\begin{prop} For any vector bundle $E$ over $X$, there exists a trivialized stratification for $E$.
\end{prop}

\begin{rem} \label{remtree} Let $\Sigma \equiv X_0 \overset{f_0}{\longrightarrow} X_1 \overset{f_1}{\longrightarrow}  ... \overset{f_{n-1}}{\longrightarrow} X_n = X$ be a stratification. It has an naturally associated \emph{tree}, with vertices indexed by irreducible components of the $X_i$. If $v$ is such a vertex, indexed by an irreducible component $V \subseteq X_i$, then the children of $v$ are indexed by the irreducible components $W$ of $X_{i-1}$ such that $f_{i-1} (W) \subseteq V$. The \emph{leaves} are in bijection with the points of $X_0$.
More precisely, this tree $\mathcal T$ satisfies the following properties:

\begin{enumerate} \itemsep=0em
\item to each node $\nu$ of $\mathcal T$ is associated an irreducible variety $V_\nu$;
\item to each arrow $\nu \longrightarrow \mu$ in $\mathcal T$ is associated a morphism $V_\nu \longrightarrow V_\mu$;  
\item if $C_\mu$ is the set of children of $\mu$, the natural map $c_\mu : \sqcup_{\nu \in C_\mu} V_{\nu} \longrightarrow V_{\mu}$ is a stratum of $V_\mu$ accordingly to Definition \ref{defistratum}; 
\item the root of $\mathcal T$ is indexed by $X$, and the leaves are indexed by points.
\end{enumerate}
Conversely, we see right away that the data of a tree satisfying the four conditions above is equivalent to the data of a stratification of $X$. This alternative description will be quite convenient for us to describe intersection computations.
\end{rem}

The following definition is useful to construct stratifications adapted to several line bundles at once.

\begin{defi} \label{defirefinement} We say that a stratum $f_1' : X_1' \longrightarrow X$ is a \emph{refinement} of another stratum $f_1 : X_1 \longrightarrow X$, if $X_1$ is isomorphic to a direct sum of components of $X_1'$, and if $f_1$ factors through $f_1'$, via the natural map $X_1 \hookrightarrow X_1'$. 

Let $\Sigma$ and $\Sigma'$ be stratifications of $X$, with their associated trees $\mathcal T$ and $\mathcal T'$. We say that $\Sigma'$ is a \emph{refinement} of $\Sigma$, if there is a embedding of trees $\varphi : \mathcal T \hookrightarrow \mathcal T'$, sending root on root, such that the following hold: with the notations of Remark \ref{remtree}, we require that if $\mu \in \mathcal T$, then $V_{\varphi(\mu)} \cong V_\mu$, and the stratum $c_{\varphi(\mu)}$ is a refinement of $c_{\mu}$.

With the same notations, we say that $\underline{\Sigma'} = (\Sigma', \mathbf{e}')$ refines $\underline{\Sigma} = (\Sigma, \mathbf{e})$ if $\Sigma'$ refines $\Sigma$, and if for any $\mu \in \mathcal T$, the trivializations $e$ and $e'$ on $V_{\mu}$ and $V_{\varphi(\mu)}$ correspond under the identification $V_\mu \cong V_{\varphi(\mu)}$.
\end{defi}

Informally, $\Sigma'$ is a refinement of $\Sigma$ if it is obtained from it by adding more boundary components to the successive strata. It is then easy to prove the following.

\begin{prop} \label{proprefine} Let $E, F \longrightarrow X$ be vector bundles, and let $\Sigma$ be a stratification adapted to $E$. There exists a stratification $\Sigma'$, refining $\Sigma$, and adapted to both $E$ and $F$.
\end{prop}

\subsection{Truncated first Chern classes} \label{sectruncatedchern}

We now define a cycle group using the data of a stratification on $X$. We choose to use rational coefficients: this will be well suited to prove general Morse inequalities holding also for $\mathbb Q$-line bundles.

\begin{defi} \label{deficyclegroup} Let $\Sigma \equiv X_0 \overset{f_0}{\longrightarrow} X_1  ... \overset{f_n}{\longrightarrow} X_n = X $ be a stratification of $X$. For each $k$ ($0 \leq k \leq n$), the \emph{$k$-th cycle group} of $\Sigma$ is the free abelian group
$$
	Z_k^\Sigma(X)_{\mathbb Q} = \bigoplus_{V \subseteq X_k} \mathbb Q \cdot [V],
$$
	where $V$ runs among the connected components of $X_k$. The \emph{total cycle group} of $\Sigma$ is the direct sum $Z_\bullet^\Sigma (X)_{\mathbb Q} = \bigoplus_{0 \leq k \leq n} Z_k^\Sigma(X)_{\mathbb Q}$.
\end{defi}

Note that for any stratification $\Sigma$, there are natural maps $Z_k^\Sigma(X)_{\mathbb Q} \overset{\rho_k}{\longrightarrow} Z_k(X)_{\mathbb Q}$, induced by the morphisms $f_i$ appearing in the stratification $\Sigma$. Here $Z_k(X)_{\mathbb Q}$ denotes the $k$-th cycle group of $X$ (see e.g. \cite{fulton98}).

\medskip

For any $\mathbb Q$-line bundle $L \longrightarrow X$, and any trivialized stratification $\underline{\Sigma} = (\Sigma, \frac{1}{d} \mathbf e)$ adapted to $L$, we will now construct a \emph{truncated first Chern class} as a particular endomorphism of $Z_\bullet^\Sigma(X)_{\mathbb Q}$. Let us first describe the elementary step of this construction.
\medskip

\begin{bfseries}Elementary step.\end{bfseries} Let $Y$ be a complex projective variety of dimension $n$, and let $Y_1 \overset{f_1}{\longrightarrow} Y$ be a stratum. Let $L$ be a $\mathbb Q$-line bundle over $Y$. Assume that a power $L^{\otimes d}$ is a standard line bundle admitting a trivialization $e$ on the open affine subset $U_1 = Y \setminus f_1(Y_1)$. Then $e$ can be seen as a meromorphic section of $L^{\otimes d}$ on $Y$: as such, it has a well defined multiplicity $m_i \in \mathbb Z$ at the generic point of any component $D_i$ of the boundary divisor underlied by $Y \setminus U_1$. Remark that by definition of the first Chern class, we have:
\begin{equation} \label{eqdefimult}
	c_1(L) \cap [Y] = \sum_i \frac{m_i}{d} [D_i] \hspace{0.5cm}  \text{in} \; A_{n-1}(Y)_{\mathbb Q}
\end{equation}
where $A_{n-1}(Y)_{\mathbb Q}$ denotes the $(n-1)$-th Chow group of $Y$ with rational coefficients (see e.g. \cite{fulton98}).

For each $i$, we let $V_i \subseteq Y_1$ be the unique connected component such that $f_1(V_i) = D_i$ (see Definition \ref{defistratum}, \eqref{itembij}). Let $l \in \left\{0, 1 \right\}$. We define a $(n-1)$-cycle in $Y_1$, as follows
\begin{equation} \label{eqchernelem}
	c_1(L, f_1)_{[l]} \cap [X] = \sum_{(-1)^l m_i > 0} \frac{m_i}{d} [V_i] \hspace{0.5cm} \text{in} \; Z_{n-1}(Y_1)_{\mathbb Q}.
\end{equation}
The previous sum is designed so that $V_i$ runs among the connected components of $Y_1$ such that $m_i$ has the same sign as $(-1)^l$. 
\medskip

\begin{defi} \label{defistratfirstchern} Let $l \in \left\{0, 1 \right\}$. The \emph{truncated first Chern class of level $l$} of $(L, \underline{\Sigma})$ is the endomorphism $c_1(L, \underline{\Sigma})_{[l]}$ of $Z^{\Sigma}_{\bullet}(X)_{\mathbb Q}$, whose action on the pure cycles $[V] \in Z_\bullet^{\Sigma}(X)_{\mathbb Q}$ is defined as follows.
Write $\Sigma \equiv X_0 \overset{f_0}{\longrightarrow} ... \overset{f_{n-1}}{\longrightarrow} X_n = X$. Let $V$ be a connected component of $X_k$, and let $g_k : f_k^{-1}(V) \longrightarrow V$ be the morphism induced by $f_k$. We see immediately from Definition \ref{defistratum} that $g_k$ is a stratum of $V$. For $l \in \left\{0, 1\right\}$, we then define, following \eqref{eqchernelem}:
\begin{equation} \label{eqstratfirstchernclass}
	c_1(L, \underline{\Sigma})_{[l]} \cap [V] = c_1(L, g_k)_{[l]} \cap [V] \;\; \in \;\;  Z_{k-1}(f_k^{-1}(V))_{\mathbb Q}.
\end{equation}
	By Definition \ref{deficyclegroup}, $Z_{k-1}(f_k^{-1}(V))_{\mathbb Q}$ is a direct summand of $Z_{k-1}^{\Sigma}(X)_{\mathbb Q}$. This permits to see the above cycle as an element of $Z^{\Sigma}_{k-1}(X)_{\mathbb Q}$.

If $l \notin \left\{ 0, 1 \right\}$, we extend the definition by $c_1(L, \underline{\Sigma})_{[l]} = 0$.

\end{defi}

\begin{defi} \label{defitrunchigherpowers}
	Let $k$ and $l$ be integers. We let $c_1(L, \underline{\Sigma})^k_{[l]}$ be the endomorphism of $Z_\bullet^\Sigma(X)_{\mathbb Q}$ defined inductively as follows.
\begin{itemize}
\item If $k = 1$, we let $c_1(L, \underline{\Sigma})^k_{[l]} = c_1(L, \underline{\Sigma})_{[l]}$.

\item If $k > 1$, we let
\begin{equation} \label{eqinductform}
c_1(L, \underline{\Sigma})^k_{[l]} = c_1(L, \underline{\Sigma})^{k-1}_{[l]} \cdot c_1(L, \underline{\Sigma})_{[0]} + c_1(L, \underline{\Sigma})^{k-1}_{[l-1]} \cdot c_1(L, \underline{\Sigma})_{[1]}
\end{equation}
\end{itemize}
\end{defi}

Note that the definition above provides $c_1(L, \underline{\Sigma})^k_{[l]} = 0$ if $l \notin \llbracket 0, k \rrbracket$.
\medskip

The Morse inequalities of Theorem \ref{thmmorse} will be stated in terms of the truncated Chern classes, as in the following definition.

\begin{defi} \label{deficherntrunc} Let $l$ and $k$ be integers $(0 \leq l \leq k)$. The \emph{$l$-truncated $k$-th power of the first Chern class} of $(L, \underline{\Sigma})$ is the following endomorphism of $Z_\bullet^\Sigma(X)_{\mathbb Q}$:
$$
c_1(L, \underline{\Sigma})^k_{[\leq l]} = \sum_{j = 0}^l c_1(L, \underline{\Sigma})^k_{[j]}.
$$
\end{defi}

The terminology of \emph{truncated} power can be justified easily as shown by the following proposition, which is easy to show by induction on $k$.

\begin{prop} Let $\rho : Z_\bullet^\Sigma(X)_{\mathbb Q} \longrightarrow Z_\bullet(X)_{\mathbb Q}$ be the natural map. For any $k$ $(0 \leq k \leq n)$, the $(n-k)$-cycle
$$
	\rho \left( \; c_1(L, \underline{\Sigma})^k_{[\leq k]} \cap [X] \; \right) \; \; \in Z_{n-k}(X)_{\mathbb Q}
$$ 

	is a representative of the cycle class $c_1(L)^k \cap[X] \in A_{n-k}(X)_{\mathbb Q}$.

\end{prop}

As usual, we have a degree map that compute the total multiplicity of a $0$-cycle.

\begin{defi} Let $\rho : Z_\bullet^\Sigma (X)_{\mathbb Q} \longrightarrow Z_\bullet (X)_{\mathbb Q}$ be the natural map. The \emph{degree map}
$$
	\mathrm{deg} : Z_0^\Sigma(X)_{\mathbb Q} \longrightarrow \mathbb Q
$$
	is the composition $\mathrm{deg}_X \circ \rho$, where $\mathrm{deg}_X : Z_0(X)_{\mathbb Q} \longrightarrow \mathbb Q$ is the usual degree map, given the sum of all multiplicities on the components of a $0$-cycle.

\end{defi}

The next lemma provides a simple formula that will be prove useful in the upcoming discussion. Consider a trivialized stratification $\underline{\Sigma} = (\Sigma, \frac{1}{d} \mathbf e)$ for a $\mathbb Q$-line bundle $L$ on $X$, and let $X_{n-1}$ the strata of dimension $n-1$ appearing in $\Sigma$, with natural map $f : X_{n-1} \longrightarrow X$. Let $e$ be the trivialization of $L^{\otimes d}$ on $X \setminus f(X_{n-1})$ provided by $\mathbf e$. Let $(D_k)$ be the family of components of $X_{n-1}$, and $m_k$ be the multiplicities of $e$ along the images of these components by $f$. For each $k$, $\underline{\Sigma}$ induces by restriction a natural trivialized stratification on $D_k$, that we denote by $\underline{\Sigma}_k$. The following proposition tells us that the truncated intersection numbers can be computed inductively using the data of the $m_k$ and of the $\underline{\Sigma_k}$.

\begin{lem} \label{lemformulainduct} With the previous notations, denote by $q_k : D_k \longrightarrow X$ the natural maps. Then, we have, for any $l \in \llbracket 0 , n \rrbracket$:
$$
	\deg c_1(L, \underline{\Sigma})^n_{[\leq l]} = \sum_{\overset{k}{m_k > 0}} \frac{m_k}{d} \deg c_1(q_k^\ast L, \underline{\Sigma_k})^{n-1}_{[\leq l]} + \sum_{\overset{k}{m_k < 0}} \frac{m_k}{d} \deg c_1(q_k^\ast L, \underline{\Sigma_k})^{n-1}_{[\leq l - 1]}.
$$ 
\end{lem}
The proof of Lemma \ref{lemformulainduct} is straightforward using the formula \eqref{eqinductform} and Definition \ref{deficherntrunc}.
\medskip

Keeping the same notations as above, let $\mathcal T$ be the tree associated to $\Sigma$ (see Remark \ref{remtree}). The trivialization $\frac{1}{d}\mathbf{e}$ provides us with a marking of all edges of $\mathcal T$ by a rational number, as follows. If $D_k$ is an irreducible component of $X_{n-1}$, labeling a vertex $\mathbf{v}_k \in \mathcal T$, we mark the edge $\mathbf{v}_k \longrightarrow \mathbf{r}$ by $\mu_k = \frac{m_k}{d}$ (where $\mathbf{r}$ is the root of $\mathcal T$). Then, we mark inductively the edges of all the trees based at the $\mathbf{v}_k$, using the trivialized stratifications $\underline{\Sigma}_k$ on each $D_k$.

\begin{defi} Let $\sigma$ be a complete path in $\mathcal T$, i.e. a path leading from the root $\mathbf{r}$ to one of the leaves of $\mathcal T$. We say that $\sigma$ is a path of \emph{index} $l$, if there are \emph{exactly} $l$ negative markings on the edges of $\sigma$.
\end{defi}

Then, by Lemma \ref{lemformulainduct}, and the construction of the markings on $\mathcal T$, the following proposition is straightforward.

\begin{prop} \label{propgraph} Let $\underline{\Sigma} = (\Sigma, \frac{1}{d} \mathbf{e})$ be as above, and let $\mathcal T$ be the tree associated to $\Sigma$, with the markings provided by the trivialization $\frac{1}{d} \mathbf{e}$. For any complete path in $\mathcal T$, let $C_\sigma$ denote the product of the markings along the edges of $\sigma$. Then, for all $l$, we have
$$
\deg c_1(L, \underline{\Sigma})^n_{[l]} = \sum_{\mathrm{index}(\sigma) = l} C_\sigma,
$$
where, in the sum above, $\sigma$ runs among all complete paths of $\mathcal T$ of index $l$.
\end{prop}

Let us finish this section with a proposition showing that we can use arbitrarily refinements of a stratification $\Sigma$ to compute the truncated Chern classes.

\begin{prop} \label{proprefine2} Let $\underline{\Sigma} = (\Sigma, \frac{1}{d} \mathbf{e})$ be a trivialized stratification, and let $\underline{\Sigma'} = (\Sigma, \frac{1}{d} \mathbf{e'})$ be a refinement of $\underline{\Sigma}$. Then, for any $l$, we have
$$
\deg c_1(L, \underline{\Sigma})^n_{[l]} = \deg c_1(L, \underline{\Sigma}')^n_{[l]}.
$$  
\end{prop}
\begin{proof}
With the notations of Definition \ref{defirefinement}, we have an embedding of trees $\varphi : \mathcal T \longrightarrow \mathcal T'$ preserving the markings. Let $\mathbf{v} \overset{\mathfrak{s}}{\longrightarrow} \mathbf{w}$ be an edge of $\mathcal T'$ which does not belong to $\varphi(\mathcal T)$, and satisfying $\mathbf{w} \in \mathcal T$. We show that $\mathfrak{s}$ is given the multiplicity $0$ for the marking associated to $\frac{1}{d} \mathbf{e'}$. Indeed, if $V$ (resp. $W$) is the irreducible variety labeling $\mathbf{v}$ (resp. $\mathbf{w}$), then $V$ does not appear among the irreducible components of the stratum of $W$ given by $\Sigma$. Thus, if $e$ is the trivialization of $L^{\otimes d}$ given by $\mathbf{e}$ near the generic point of $W$, then $e$ must be invertible near the generic point of $\mathrm{Im}(V \longrightarrow W)$. This shows that $\mathfrak{s}$ is marked with $0$.  Hence, if $\sigma$ is a complete path not included in $\varphi(\mathcal T)$, it has one edge $\mathfrak{s}$ as above, and we have $C_\sigma = 0$. Hence, this path does not contribute to the sum defining $\deg c_1(L, \underline{\Sigma}')^n_{[l]}$ as in Proposition \ref{propgraph}. This ends the proof.
\end{proof}

\subsection{Morse inequalities} \label{sectmorseineq}

As before, let $X$ be a normal complex projective variety of dimension $n$.

\begin{defi} Let $\mathcal F$ be a coherent sheaf on $X$, and let $l$ be an integer ($0 \leq l \leq n$). The \emph{$l$-th truncated} Euler characteristic of $\mathcal F$ is the integer
$$
\chi^{[l]} (X, \mathcal F) = \sum_{j = 0}^l (-1)^{j+l} h^j(X, \mathcal F).
$$
\end{defi}

Note that the top $h^j$ appearing in this definition comes with a positive sign. We are now in position to state and prove the following algebraic Morse inequalities.

\begin{thm}[Morse inequalities] \label{thmmorse} Let $L$ be a $\mathbb Q$-line bundle on $X$, and let $\underline{\Sigma} = (\Sigma, \frac{1}{d} \mathbf e)$ be a trivialized stratification for $L$. Let $M$ be another line bundle on $X$. Then, for each integer $i$, and for any $m$ divisible by $d$, we have
	\medskip

\begin{enumerate}[(i)]
\item (Strong Morse inequalities) 
$$
\chi^{[i]} (X, M \otimes L^{\otimes m}) \leq (-1)^i \left( \deg c_1(L, \underline{\Sigma})^n_{[\leq i]} \right) \frac{m^n}{n!} + O(m^{n-1})
$$

\item (Weak Morse inequalities) 
$$
h^{i} (X, M \otimes L^{\otimes m}) \leq (-1)^i \left( \deg c_1(L, \underline{\Sigma})^n_{[i]} \right) \frac{m^n}{n!} + O(m^{n-1})
$$

\item (Asymptotic Riemann-Roch formula)
$$
\chi(X, M \otimes L^{\otimes m}) = \left( \deg c_1(L, \underline{\Sigma})^n_{[\leq n]} \right) \frac{m^n}{n!} + O(m^{n-1})
$$
\end{enumerate}
\end{thm}
	
	\begin{rem} We will first prove this result under the assumption that $L$ is a \emph{standard line bundle} (but we nevertheless assume that $\underline{\Sigma} = (\Sigma, \frac{1}{d} \mathbf{e})$, where $\mathbf{e}$ is a trivialization of $L^{\otimes d}$ on $\Sigma$, for some $d \geq 1$). We will prove the general case in Section \ref{sectblochgieseker}, after some considerations about Bloch-Gieseker coverings.
	\end{rem}

\begin{proof}[Proof (Case where $L$ is a standard line bundle)]

As usual, the weak Morse inequalities follow from the strong ones, remarking that $h^{i} = \chi^{[i]} + \chi^{[i-1]}$. Also, the asymptotic Riemann-Roch formula can be obtained from the strong Morse inequalities, using $\chi = (-1)^n \chi^{[n]} = (-1)^{n+1} \chi^{[n+1]}$. Thus, it suffices to prove the first point. 
\medskip

We will actually prove the seemingly more general proposition.

\begin{prop} \label{propredmorse} Under the hypotheses of Theorem \ref{thmmorse}, let $p: X' \longrightarrow X$ be a proper generically finite morphism of degree $D$, and let $M'$ be any line bundle on $X'$. Then, for each $i$, and any $m$ divisible by $d$, we have
\begin{equation} \label{eqmorseupperboundbir}
\chi^{[i]} (X', M' \otimes p^\ast L^{\otimes m}) \leq D (-1)^i \left( \deg c_1(L, \underline{\Sigma})^n_{[\leq i]} \right) \frac{m^n}{n!} + O(m^{n-1}).
\end{equation}
\end{prop}

Let us now prove Proposition \ref{propredmorse} by induction on $\dim X$.
\medskip

	\emph{Step 1. Initialization of the induction.} If $\dim X = 0$, then $X'$ is just a union of $D$ scheme points. In this case, for each $i \geq 0$ and each $m$ divisible enough, we have $\chi^{[i]} (X', M' \otimes p^\ast L^{\otimes m}) = (-1)^i h^0 (X', M' \otimes p^\ast L^{\otimes m}) = (-1)^i D$, and $\deg c_1(L, \underline{\Sigma})^0_{[\leq 0]} \cap [X] = \deg [X] = 1$, so the result holds in this case (with $O(m^{-1}) = 0$).

	Suppose now that the result has been proved for $\dim X \leq n-1$. Write $\Sigma \equiv X_0 \overset{f_0}{\longrightarrow} ... X_{n-1} \overset{f_{n-1}}{\longrightarrow} X_n = X$, and let $e$ be the trivialization of $L^{\otimes d}$ over the open affine subset $U_n = X_n \setminus f_{n-1}(X_{n-1})$ given by $\mathbf e$.  
\medskip

	Let $D_n = \sum_{1 \leq j \leq r} m_j F_j \in Z_{n-1}(X)$ be the \emph{Weil} divisor of zeros and poles of the section $e$, seen as a meromorphic section of $L^{\otimes d}$. For each $j$ ($1 \leq j \leq r$), let $V_j$ be the unique connected component of $X_{n-1}$ such that $f(V_j) = F_j$, and let $q_j : V_j \longrightarrow X$ be the natural map.
\medskip

	\emph{Step 2. We pass to a simpler birational model of $X'$.} By Lemma \ref{lemmodification}, it suffices to prove Proposition \ref{propredmorse} with $X'$ replaced by any modification $q : \widetilde{X'} \longrightarrow X'$, $M'$ replaced by $q^\ast M'$, and $p$ replaced by $p \circ q$. Thus, we can suppose without loss of generality that $X'$ is a \emph{smooth} complex projective manifold, and that the inverse image $p^{-1} (D_n)$ is a divisor with simple normal crossings.
By construction, the line bundle $p^\ast L$ is trivialized by the section $p^\ast e$ on the open dense subset $p^{-1}(U_n)$. 

Thus, we have the following \emph{Cartier divisors} identity on $X'$:
$$
	D(p^\ast e) = \sum_{1 \leq j \leq s} m_j' F_j' + \sum_{1 \leq k \leq s} a_k E_k' 
$$
	where each $F_j'$ is a Cartier divisor on $X'$ dominating $F_j$ for some $1 \leq j \leq r$, and each $E_k'$ is a $p$-exceptional divisor, weighted with a coefficient $a_k \in \mathbb Z$. Note that if $m_j$ is the multiplicity of $e$ along $F_j$, then the projection formula yields, for all $j$: 
	\begin{align} \label{eqsumproj1}
		\sum_{k} m_k' \, \deg (F_k' / F_j) & = D\, m_j,	\end{align}
where the sum runs among all $k$ such that $F_k'$ dominates $F_j$.

	\emph{Step 3. We use $e$ to get a cyclic cover of $X'$}. The trivialization $e$ can be seen as a meromorphic section of $p^\ast L^{\otimes d}$ on the open dense subset $p^{-1}(U_n)$. It permits to form a cyclic cover $X_c \longrightarrow X'$ (by definition, $X_c$ is the locus $ \{ [1 : x^d] = [1 : p^\ast e ] \} \subseteq \mathbb P_{X'} (\mathcal O_X \oplus p^\ast L)$). Let $X_c^n$ be the normalization of $X_c$. We can again resolve the singularities of $X_c^{n}$ be performing blowing-ups on centers projecting on the intersections of the different $F_j'$, to obtain a smooth manifold $X''$. Let $r : X'' \longrightarrow X'$ and $q : X'' \longrightarrow X$ be the natural maps. For each $j$, we can write $r^\ast F_j' = \sum_{1 \leq l \leq n_j} d_{l} F_{j,l}'' + \sum_l b_l E_k''$, where the $E_k''$ are $r$-exceptional divisors. The line bundle $q^\ast L$ has a canonical meromorphic section $e''$ (given by the pullback of the section $x$ of $\mathcal O(1) \longrightarrow \mathbb P_{X'} (\mathcal O_X \oplus p^\ast L)$. We have then $(e'')^{\otimes d} = q^\ast e$ as a meromorphic section  of $q^\ast L^{\otimes d}$.

	We can write the divisor associated to this section $e''$ as follows:
	$$
	D(e'') = \sum_{1 \leq j \leq r} \sum_{1 \leq l \leq n_j} m''_{j, l} F_{j,l}'' + \sum_{1 \leq k \leq r} c_k E_{k}'',
	$$
	where the $E_{k}''$ are $q$-exceptional. Note that for any $l$, all $m''_{j, l}$ have the same sign as $m_l$ (this sign is determined by whether or not $p^\ast e$ is regular near $F_j'$). Also, the projection formula gives $d \, r_\ast D(e'') = r_\ast D(r^\ast p^\ast e) = d \, D(p^\ast e)$. This yields 
	\begin{equation} \label{eqsumproj2}
		m_j' = \sum_{1 \leq l \leq n_j} m''_{j, l} \, \mathrm{deg}(F_{j,l}''/F_j')
	\end{equation} In the following, we let $M'' = r^\ast M'$.

\medskip

	\emph{Step 4. We bound from above the difference of two consecutive $\chi^{[i]}(M'' \otimes q^\ast L^{\otimes m})$}. Letting $A = \sum_{m_{j,l}'' > 0} m_{j, l}'' F_{j,l}'' + \sum_{c_k > 0} c_k E''_k$, and $B = \sum_{m_{j, l} < 0} (-m_{j, l}'') F_{j,l}'' + \sum_{c_k < 0} c_k E_k''$, we obtain the two exact sequences

\begin{tikzcd}[column sep=1em]
 0 \ar[r]  & M'' \otimes q^\ast L^{\otimes(m+1)} \otimes \mathcal O_{X''}( - A) \ar[r] \ar[d, equal] & M'' \otimes q^\ast L^{\otimes(m+1)} \ar[r] &  M'' \otimes q^\ast L^{\otimes(m+1)} \otimes \mathcal O_{A} \ar[r] &  0 \\
 0 \ar[r] & M'' \otimes q^\ast L^{\otimes m} \otimes \mathcal O_{X''}( - B) \ar[r] & M'' \otimes q^\ast L^{\otimes m} \ar[r] &  M'' \otimes q^\ast L^{\otimes m} \otimes \mathcal O_{B} \ar[r] &  0 
\end{tikzcd}

Taking an initial fragment of the long exact sequence associated to the first line, we obtain
\begin{align*}
\cdots \longrightarrow H^i (M'' \otimes & \, q^\ast L^{\otimes(m+1)}( - A))  \longrightarrow  \\
&  H^i (M'' \otimes q^\ast L^{\otimes(m+1)}) \longrightarrow H^i (M'' \otimes q^\ast L^{\otimes(m+1)} \otimes \mathcal O_{A}) \longrightarrow Z \longrightarrow 0.
\end{align*}
where the cohomology groups are taken over $X''$. Since $\dim Z \geq 0$, taking the Euler characteristic yields
\begin{align*}
0 \leq \dim Z  =  - & \chi^{[i]} (X'',  M'' \otimes q^\ast L^{\otimes(m+1)}) \\
& + \chi^{[i]} (X'', M'' \otimes q^\ast L^{\otimes(m+1)} \otimes \mathcal O_{X''}(-A)) + \chi^{[i]}(X'', M'' \otimes q^\ast L^{\otimes(m+1)} \otimes \mathcal O_A )
\end{align*}
Similarly, the second line yields:
\begin{align*}
0 \leq  \chi^{[i]}& (X'',  M'' \otimes q^\ast L^{\otimes m}) \\
& - \chi^{[i]} (X'', M'' \otimes q^\ast L^{\otimes m} \otimes \mathcal O_{X''}(-B)) + \chi^{[i - 1]}(X'', M'' \otimes q^\ast L^{\otimes m} \otimes \mathcal O_B )
\end{align*}
Summing these two equations, we obtain
\begin{align} \label{eqtelescopic}
\chi^{[i]} (X'',  M'' \otimes q^\ast L^{\otimes(m+1)})  - & \chi^{[i]} (X'',  M'' \otimes q^\ast L^{\otimes m}) \leq \\ \nonumber
& \chi^{[i]}(X'', M'' \otimes q^\ast L^{\otimes(m+1)} \otimes \mathcal O_A ) + \chi^{[i - 1]}(X'', M'' \otimes q^\ast L^{\otimes m} \otimes \mathcal O_B )
\end{align}

\emph{Step 5. We write an upper bound on the right hand side of \eqref{eqtelescopic}}. Since we have the identity of \emph{Cartier} divisors $A = \sum_{m_{j,k}'' > 0} m_{j, k}'' F''_{j,k} + \sum_{c_k > 0} c_k E''_k$, and each connected component of $X''$ is integral (since $X''$ is normal), the sheaf $\mathcal O_A$ admits a filtration $0 \subseteq \mathcal F_0 \subseteq ... \subseteq \mathcal{F}_{N} = \mathcal O_A$, where the graded terms are given by short exact sequences
$$
0 \longrightarrow \mathcal F_j \longrightarrow \mathcal F_{j+1} \longrightarrow \mathcal O_{C} \otimes \left( N_C^\ast \right)^{l} \longrightarrow 0.
$$
where $C$ is one of the reduced components of $A$, $N_C$ is the normal bundle to $C$, and where $0 \leq l \leq \mathrm{mult}_A(D)$ (see Lemma \ref{lemfilt}). Taking the long exact sequences in cohomology and the truncated Euler characteristic, we obtain:
$$
\chi^{[i]} (\mathcal F_{j+1} \otimes q^\ast L^{(m+1)} \otimes M'') \leq \chi^{[i]} (\mathcal F_{j} \otimes q^\ast L^{(m+1)} \otimes M'') + \chi^{[i]}(C, \left( N_C^\ast \right)^{l} \otimes (q^\ast L^{(m+1)} \otimes M'')|_C).
$$
After summation, this gives
\begin{align} \nonumber 
	\chi^{[i]}(X'', M'' \otimes q^\ast L^{\otimes(m+1)} \otimes \mathcal O_A ) & \leq \sum_{j,k} \; \sum_{1 \leq l \leq m''_{j,k}} \chi^{[i]}(F_{j,k}'', \left( N_{F_{j,k}''}^\ast \right)^{l} \otimes q^\ast L^{\otimes (m+1)} \otimes M''') \\ \label{eqstepA} 
 & + \sum_k \sum_{1 \leq l \leq c_k} \chi^{[i]}(E_k, \left( N_{E_k}^\ast \right)^{l} \otimes q^\ast L^{(m+1)} \otimes M'')
\end{align}
where the first (resp. second) sum runs through the indexes $j, k$ (resp. $k$) such that $m_{j,k}'' >0$ (resp. $c_k > 0$). Now, a standard argument shows that for any $1 \leq l \leq m_{j,k}''$, we have, as $m \longrightarrow + \infty$, 
$$
\chi^{[i]}(F_k'', \left( N_{F_{j,k}''} ^\ast \right)^{l} \otimes q^\ast L^{(m+1)} \otimes M'') = \chi^{[i]}(F_k'', q^\ast L^{\otimes (m+1)} \otimes M'') + O(m^{n-2})
$$

Besides, since each $E_k''$ is $q$-exceptional, we have, for any $1 \leq l \leq c_k$, as $m \longrightarrow + \infty$:
$$
\chi^{[i]}(E_k'', \left( N_{E_k''}^\ast \right)^{l} \otimes q^\ast L^{(m+1)} \otimes M'') = O(m^{n-2}).
$$

Inserting this in \eqref{eqstepA}, we obtain
$$
\chi^{[i]}(X'', M'' \otimes q^\ast L^{\otimes(m+1)} \otimes \mathcal O_A ) \leq \sum_{j,k} m''_{j,k}\,   \chi^{[i]}(F_k'', q^\ast L^{(m+1)} \otimes M'') + O(m^{n-2}).
$$
where the sum runs among all $j,k$ such that $m_{j,k}'' > 0$. Similarly,
$$
\chi^{[i-1]}(X'', M'' \otimes q^\ast L^{\otimes m} \otimes \mathcal O_B ) \leq \sum_{j,k} (-m_{j,k}'')\,   \chi^{[i- 1]}(F_k'', q^\ast L^{ m} \otimes M'') + O(m^{n-2}).
$$
where the sum runs among all $k$ such that $m_{j,k}'' < 0$. 

\emph{Step 6. We write the final upper bound.} Putting the last two equations in \eqref{eqtelescopic}, and summing over $m$, we finally obtain 
\begin{align*}
 \chi^{[i]} ( & X'',   M'' \otimes   q^\ast L^{\otimes(m+1)})  \leq \\ 
	& \sum_{l \leq m} \left(\sum_{\overset{j,k}{m_{j,k}'' > 0}} m_{j,k}''\,   \chi^{[i]}(F_{j,k}'', q^\ast L^{\otimes (l+1)} \otimes M'') +  \sum_{\overset{j,k}{m_{j,k}'' < 0}} (-m_{j,k}'')\,   \chi^{[i- 1]}(F_{j,k}'', q^\ast L^{ \otimes l} \otimes M'') + O(l^{n-2}) \right),
\end{align*}
where the constants appearing in the $O(l^{n-2})$-terms do not depend on $m$. Since the restrictions $q : F_{j,k}'' \longrightarrow F_l$ are finite dominant morphisms, we can now apply the induction hypothesis to each $F_{j,k}''$. For any $j$, let $\underline{\Sigma}_l$ be the trivialized stratification induced on $V_l$ by $\underline{\Sigma}$. Using the induction hypothesis, we get

\begin{align*}
 \chi^{[i]} ( & X'',   M'' \otimes   q^\ast L^{\otimes(m+1)})  \leq \\ 
	& \sum_{l \leq m} \left(\sum_{\overset{j,k}{m_{j,k}'' > 0}} m_{j,k}''\,  \mathrm{deg}(F_{j,k}''/F_l) \; c_1(q_l^\ast L, \underline{\Sigma_l})^n_{[\leq i]} +  \sum_{\overset{j,k}{m_{j,k}'' < 0}} (-m_{j,k}'')\,  \mathrm{deg}(F_{j,k}''/F_l) \; c_1(q_l^\ast L, \underline{\Sigma_l})^n_{[\leq i - 1]} \right)  \\
	& \hspace{5em}+ O(l^{n-2}),
\end{align*}
where for all $j,k$, we have written $l$ to denote the index of the component $V_l$ of $D_n$ dominated by $F_{j,k}''$.

Now, \eqref{eqsumproj1} and \eqref{eqsumproj2} give $\sum_{j,k} {m_{j,k}''} \mathrm{deg}(F_{j,k}''/F_l) = D m_l$, where the sum runs among all $j,k$ such that $F_{j,k}''$ dominates $F_l$. Thus, since
$$	
	\chi^{[i]}  (X',   M' \otimes p^\ast L^{\otimes(m+1)})  = \frac{1}{d} \chi^{[i]} (X'', M'' \otimes q^\ast L^{\otimes {m+1}}) + O(m^{n-1}),
$$
by Lemma \ref{lemmodification}, we obtain
	\begin{align*}
		\chi^{[i]} &  (X',   M' \otimes p^\ast L^{\otimes(m+1)})  \\
	& \leq  D \, (-1)^i \left( \sum_{\overset{l}{m_l > 0}} \frac{m_l}{d}\, \deg c_1(q_l^\ast L, \underline{\Sigma_k})_{[\leq i]} -  \sum_{\overset{k}{m_k < 0}} (-\frac{m_k}{d})\,   \deg c_1{}(q_k^\ast L, \underline{\Sigma_k})_{[\leq i -1]} \right) \sum_{l \leq m} \frac{l^{n-1}}{(n-1)!} \\
&  \hspace{50pt}  + \sum_{l \leq m} O(l^{n-2}).\\
\end{align*}

Since $\sum_{1 \leq l \leq m} \frac{l^{n-1}}{(n-1)!} = \frac{m^n}{n!} + O(m^{n-1})$ and $\sum_{l \leq m} l^{n-2} = O(m^{n-1})$, the conclusion then comes immediately from Lemma \ref{lemformulainduct}.
\end{proof}

Let us mention briefly two classical lemmas that were used in the proof.

\begin{lem}  [see \cite{dem11}] \label{lemmodification} Let $X$ be a reduced complex scheme of pure dimension $n$, and let $X' \overset{p}{\longrightarrow} X$ be a dominant proper generically finite morphism of degree $d$. Then, for any line bundles $M$ and $L$ on $X$, and any $i \geq 0$, we have
$$
\chi^{[i]}(X', p^\ast M \otimes p^\ast L^{\otimes m}) = d \; \chi^{[i]}(X, M \otimes L^{\otimes m}) + O(m^{n-1}).
$$
\end{lem}

The previous lemma can be proved by a standard application of Leray's spectral sequence. It suffices to remark the following two facts: first that $E_2^{r, s} = H^s(X, R^r p_\ast \mathcal O_{X'} \otimes M \otimes L^{\otimes m})$ has $O(m^{n-1})$ dimension if $r > 0$, since then $R^r p_\ast \mathcal O_{X'}$ is a torsion sheaf on $X$. Second, we have $p_\ast \mathcal O_{X'}|_{U} \cong \mathcal O_{U}^{\oplus d}$ for a Zariski open dense subset $U \subseteq X$, so $h^i(X, p_\ast \mathcal O_{X'} \otimes  M \otimes L^{\otimes m}) = d \, h^i(X, M \otimes L^{\otimes m}) + O(m^{n-1})$.

\begin{lem} \label{lemfilt} Let $X$ be an integral complex scheme on $X$, and let $D_1, ..., D_r$ be irreducible Cartier divisors on $X$. For all $1 \leq i \leq r$, let $m_i \in \mathbb N$, and define $D = \sum_{i} m_i D_i$. Then there exists a filtration $\mathcal F_1 \subseteq ... \subseteq \mathcal F_{N} = \mathcal O_{D}$ (where $N = \sum_i m_i$), with successive quotients given as follows:
$$
0 \longrightarrow \mathcal F_i \longrightarrow \mathcal F_{i+1}  \longrightarrow N_{D_i}^{ \otimes k_i} \longrightarrow 0,
$$
for some $1 \leq k_i \leq m_i$, where $N_{D_i}$ is the normal $\mathcal O_{D_i}$-line bundle of $D_i$.
\end{lem}
\begin{proof}
Assume first that $X = \mathrm{Spec} A$, where $A$ is an integral ring, and that each $D_i$ is given by some $f_i \in D_i$. Then, we define the $\mathcal F_i$ so that they are associated to $A$-submodules of $\quotientd{A}{(f_1^{m_1} ... f_r^{m_r})}$ of the form
$$
M_{a_1, ..., a_r} = \quotientd{(f_1^{a_1} ... f_r^{a_r})}{(f_1^{m_1} ... f_r^{m_r})}
$$
with $0 \leq a_i \leq m_i$ for all $i$. Since $A$ is integral, we see that $M_{a_1, ..., a_r} \hookrightarrow M_{b_1, ..., b_r}$ if $a_i \geq b_i$ for all $i$. Also $\quotientd{M_{a_1, ..., a_i + 1, ..., a_r}}{M_{a_1, ..., a_i, ..., a_r }} \cong \quotientd{(f_i^{a_i})}{(f_i^{a_i + 1})}$ and this last $A$-module defines the sheaf $N_{D_i}^{\otimes a_i}$, supported on $D_i$. The general case follows easily by covering $X$ with affine open subsets: the local definitions glue together.
\end{proof}

The algebraic Morse inequalities of Demailly and Angelini appear now as a particular case of Theorem \ref{thmmorse}.

\begin{thm}[Demailly \cite{dem96}, Angelini \cite{ang96}] Let $X$ be projective variety of dimension $n$, and let $L$ be a line bundle on $X$. Assume $L = \mathcal O(F - G)$, where $F, G$ are Cartier nef divisors. Then, for any $i \in \llbracket 0, n \rrbracket$, we have
$$
\chi^{[i]} (X, L^{\otimes m}) \leq \left[ \sum_{ 0 \leq j \leq i} (-1)^{i + j} \binom{n}{j} F^{n-j} \cdot G^j \right] \frac{m^n}{n!} \, + \, O(m^{n-1}).
$$
\end{thm}
\begin{proof}
As usual, we do not lose generality in assuming that $X$ smooth, by replacing it by some other birational modification. Also, replacing $F$ (resp. $G$) by $s(rF + A)$ (resp. $s(rG + A)$) with $r, s$ large, $A$ an ample divisor, we can assume that both $F$ and $G$ are very ample.

We are now going to exhibit a stratification $\Sigma$ on $X$ to which apply Theorem \ref{thmmorse}. Let $X_n = X$. Since $F$ and $G$ are very ample, we can replace them by smooth members of their linear equivalence class, and assume that $X_{n-1} = |F| + |G|$ is a simple normal crossing divisor. Let $e_n$ be a meromorphic section of $L$ which trivializes $L$ on $U_n = X_n \setminus X_{n-1}$, such that $D(e_n) = F - G$. Now, $\mathcal O(F)$ and $\mathcal O(G)$ are very ample when restricted to both $E$ and $F$, so we can iterate this construction to produce stratifications on both $E$ and $F$. This gives a sequence of strata $X_0 \longrightarrow ... \longrightarrow X_{n-2} \longrightarrow X_{n-1} = |E| \sqcup |F|$, such that, for each $i$, $L|_{X_i}$ admits a meromorphic section $e_i$ with $D(e_i) = F_i - G_i$, where $\mathcal O(F_i) = \mathcal O(F)|_{X_i}$, and $\mathcal O(G_i) = \mathcal O(G)|_{X_i}$. Putting this data together with $X_n = X$ and the trivialization $e_n$, we get the requested trivialized stratification $\underline{\Sigma}$ of $L$ over $X$.

	Now, by construction of $\underline{\Sigma}$, if $\alpha \in Z^{\Sigma}_k X_{\mathbb Q}$, the cycle class of $c_1(L, \underline{\Sigma})_{[0]} \cap \alpha$ (resp. $c_1(L, \underline{\Sigma})_{[1]} \cap \alpha$) is equal to $c_1(F) \cap \alpha$ (resp. $-c_1(G) \cap \alpha$) in $A_k(X)$. Consequently, iterating Definition \ref{defitrunchigherpowers} yields, for any $j \in \llbracket 0, n \rrbracket$: 
\begin{align*}
c_1(L, \underline{\Sigma})^n_{[j]} \cap [X] & = \sum_{\overset{T \in \llbracket 1, n \rrbracket}{|T| = j}}  c_1(L, \underline{\Sigma})_{[\mathbbm{1}_{1 \in T}]} \cap ... \cap  c_1(L, \underline{\Sigma})_{[\mathbbm{1}_{n \in T}]} \cap [X] \\
& = \sum_{\overset{T \in \llbracket 1, n \rrbracket}{|T| = j}} ( c_1(F))^{n - |T|} \cap (-c_1(G))^{|T|} \cap [X] \\
& = (-1)^{j} \binom{n}{j} F^{n- j} \cdot G^j
\end{align*}
	in $A_0(X)$ (here $\mathbbm{1}_{j \in T}$ is equal to $1$ if $j \in T$, and to $0$ otherwise). By Theorem \ref{thmmorse}, this gives the result.
\end{proof}

\subsection{Bloch-Gieseker coverings} \label{sectblochgieseker}

In this section, we explain how to lift a stratification to a Bloch-Gieseker covering, and how to compute its truncated intersection numbers.

\medskip

Consider an $n$-dimensional variety $X$ equipped with a line bundle $L$, with a trivialized stratification $\underline{\Sigma} =(\Sigma, \mathbf e)$. Let $A$ be a very ample line bundle on $X$, associated with an embedding $X \hookrightarrow \mathbb P^N$. Fix $d \in \mathbb N^\ast$, and let $\widehat{X} \overset{p}{\longrightarrow} X$ be the morphism obtained by taking the cartesian product
$$
\begin{tikzcd}
	\widehat{X} \arrow[r] \arrow[d] & \mathbb P^N \arrow[d, "q"] \\
	X \arrow[r] & \mathbb P^N
\end{tikzcd}
$$
where $q$ is the endomomorphism $[z_0 : ... : z_N] \longmapsto [z_0^d : ... : z_N^d]$ of $\mathbb P^N$. Recall that this morphism is the base step of the Bloch-Gieseker construction (cf. \cite{BG71}, see also \cite{KM98} or \cite[Theorem 4.1]{lazpos1}).
\medskip

For each irreducible variety $V$ appearing as a component of a stratum of $\Sigma$, the morphism $\widehat{X} \longrightarrow X$ induces a fibre product
\begin{equation} \label{diagV}
\begin{tikzcd}
	\widehat{V} \arrow[r] \arrow[d] & \widehat{X} \arrow[d, "p"] \\
	V \arrow[r, "f"] & X
\end{tikzcd}
\end{equation}
where $V \longrightarrow X$ is naturally induced by $\Sigma$. Also, the maps between two strata $V \longrightarrow W$ induce natural maps $\widehat{W} \longrightarrow \widehat{V}$. Putting all the maps $\widehat{W} \longrightarrow \widehat{V}$ together, we get a stratification $\widehat{\Sigma}$ on $\widehat{X}$. We will call it the \emph{pull-back stratification} of $\Sigma$ by $p$. We can also pull back the stratifications of $L$ provided by $\mathbf{e}$ on the strata of $\widehat{\Sigma}$, to get a trivialization $\widehat{\mathbf{e}}$ of $p^\ast L$. 

Then, we have a projection formula, as follows.

\begin{lem} \label{lemBG}
	Assume that $X \hookrightarrow \mathbb P^N$ is in general position. With the previous notations, let $\underline{\widehat{\Sigma}} = (\widehat{\Sigma}, \widehat{\mathbf e})$. We have then, for any $j \in \llbracket 0, n \rrbracket$:
$$
	\deg c_1(p^\ast L, \underline{\widehat{\Sigma}})^n_{[\leq j]} = \deg(p) \cdot \deg c_1(L, \underline{\Sigma})^n_{[\leq j]}.
$$
\end{lem}
\begin{proof}
	Let $\mathcal T$ be the marked tree associated to $\Sigma$. To prove this result, we are going to describe the marked tree $\widehat{\mathcal T}$ associated to $\underline{\widehat{\Sigma}}$ in terms of $\mathcal T$. Resume the notations of the diagram \eqref{diagV}. Remark that the construction above provides a natural morphism of trees $\phi : \widehat{\mathcal T} \longrightarrow \mathcal T$, sending the vertices labeled by the irreducible components of $\widehat{V}$ onto the vertex labeled by $V$.
	\medskip

	If $X \hookrightarrow \mathbb P^N$ is in general position, then each $f(V)$ intersects each of the hyperplanes $\{ z_i = 0 \} \subseteq \mathbb P^N$ transversally (in particular, this intersection is empty in the case $\dim f(V) = 0$).

	Recall, by Definitions \ref{defistratum} and \ref{defistratification}, that each $f : V \longrightarrow f(V)$ is birational onto its image. This implies that if $\dim V \geq 1$, then $\dim f(V) \geq 1$. In this case, the map $\widehat{V} \longrightarrow V$ is then a finite ramified cover of degree $\deg(p)$, and $\widehat{V}$ is irreducible. If $\dim V = 0$, then $f(V)$ is included in the locus where $p$ is étale, and then $\widehat{V}$ is a finite union of $\deg (p)$ reduced scheme points. This shows that if $\dim V \geq 1$ (resp. if $\dim V = 0$), the fibre $\phi^{-1}(V)$ contains exactly one (resp. $\deg(p)$) vertex of $\widehat{\mathcal T}$. Similarly, if $\mathfrak{s}$ is an edge of $\mathcal T$ between two varieties of dimensions $i$ and $i + 1$, then $\phi^{-1}(\mathfrak{s})$ contains $1$ (resp. $\deg(p)$) edges of $\widehat{\mathcal T}$ if $i > 0$ (resp. $i = 0$).

	\medskip

	For any $V$, the map $p$ realizes an isomorphism around the generic point of $f(V)$ in the local analytic topology (or in the \'etale topology). This has the following consequence: let $e$ be the trivialization of $L$ on some open subset $U \subseteq V$ provided by $\mathbf{e}$, and let $m$ be the multiplicity of $e$ along a strata projecting on a component $W$ of $V \setminus U$. Then, the multiplicity of the pullback $p^\ast e$ along each component of $\widehat{W}$ is also equal to $m$. This shows that if $\widehat{\mathfrak{s}}$ is an edge of $\widehat{\mathcal T}$, then $\widehat{\mathfrak{s}}$ and $\varphi(\widehat{\mathfrak{s}})$ have the same multiplicities.
	\medskip

	We have then shown that $\widehat{\mathcal T}$ can be described as the tree $\mathcal T$ where each leaf have been replaced by $\deg(p)$ copies, each edge keeping the same multiplicity. The inverse image by $\varphi$ of complete path in $\mathcal T$, consists in $\deg(p)$ paths in $\widehat{\mathcal T}$ with the same multiplicities. We can now conclude by Proposition \ref{propgraph}.
\end{proof}

We are now ready to end the proof of Theorem \ref{thmmorse} in the general case where $L$ is \emph{any} $\mathbb Q$-line bundle.

\begin{proof}[Proof of Theorem \ref{thmmorse} in the general case]
	
	We assume now that assume that $L = N^{\otimes 1/d}$ is a formal root of a standard line bundle, and we show how to prove the strong Morse inequalities. We can perform a generic Bloch-Gieseker covering $\widehat{X} \overset{p}{\longrightarrow} X$, in order to ensure that $p^\ast N$ has a $d$-th root, that we will denote by $L'$. Now, if $\underline{\widehat{\Sigma}} = (\widehat{\Sigma}, \frac{1}{d} \widehat{\mathbf{e}})$ is the pullback trivialized stratification of $\underline{\Sigma}$, we can use the version of Theorem \ref{thmmorse}, valid in the case where $L$ is a line bundle, to get
	\begin{equation} \label{eqchiBG}
		\chi^{[i]}(\widehat{X}, (N')^{\otimes m} \otimes p^\ast M) \leq (-1)^i c_1(L', \underline{\widehat{\Sigma}})^n_{[\leq i]} \frac{m^n}{n!} + O(m^{n-1}).
	\end{equation}
for $m$ divisible by $d$.

	Now, for such $m$, we have $\chi^{[i]}(\widehat{X}, (N')^{\otimes m} \otimes p^\ast M) = \chi^{[i]} (\widehat{X}, p^\ast L^{\otimes m} \otimes p^\ast M)$ and by Lemma \ref{lemmodification}, this is equal to $(\deg p) \, \chi^{[i]}(X, L^{\otimes m} \otimes M)$ up to a $O(m^{n-1})$ term. Also, by Lemma \ref{lemBG}, we have $c_1(L', \widehat{\underline{\Sigma}})^n_{[\leq i]} = \deg (p) \, c_1(L, \underline{\Sigma} )^n_{[\leq i]}$. Substituting in both sides of \eqref{eqchiBG} and dividing by $\deg (p)$, we get the requested formula.
\end{proof}

 \section{Green-Griffiths jet bundles. Preparation of the proof}

The main goal of this section is to introduce several reduction steps to simplify our proof of Theorem \ref{thmprinc}. It will also serve as a motivation for Theorem \ref{thmineqintegral}, whose scope goes beyond the study of Green-Griffiths jet differentials. 

\subsection{Weighted direct sums and Green-Griffiths vector bundles} \label{sectweightproj}

Our proof of Theorem \ref{thmprinc} will come from an application of a variant of the Morse inequalities of Theorem \ref{thmmorse} to the symmetric powers of some particular \emph{weighted direct sums of vector bundles}. We give now a few definitions related to these objects, and to the Green-Griffiths vector bundles. Most of these definitions can be found in \cite{dem12a}.
\medskip

\begin{defi} Let $X$ be a projective variety. We let $E_1, ..., E_r$ be vector bundles on $X$, and $a_1, ..., a_r$ be positive integers. We will often refer to this data as to the one of a \emph{weighted direct sum} $\mathbb E = E_1^{(a_1)} \oplus ... \oplus E_r^{(a_r)}$.

For any such weighted direct sum, we define its \emph{$m$-th symmetric product} to be the vector bundle on $X$

\begin{equation} \label{eqdefsym}
 {S}^m\, \mathbb E =  \bigoplus_{a_1 l_1 + ... + a_k l_k = m}  S^{l_1}\, E_1^{\ast} \; \otimes ... \otimes S^{l_k}\, E_k^{\ast}
\end{equation}
\end{defi}

The above terminology was used in \cite{cad17} to construct the \emph{weighted projective bundles} $\mathbb P(\mathbb E) = \mathbf{Proj}_X (S^\bullet \mathbb E)$, which were studied in particular by Al-Amrani \cite{alamrani97}. 

\medskip

Assume now that $X$ is a complex smooth projective manifold, and let $k \in \mathbb N$. We denote by $E_{k, \bullet}^{GG} \Omega_X = \bigoplus_{m \geq 0} E_{k, m}^{GG} \Omega_X$ the graded algebra of \emph{Green-Griffiths jet differentials} of order $k$ on $X$. By construction, the local holomorphic sections of $E_{k, m}^{GG} \Omega_X$ are holomorphic differential equations of order $k$ and of (weighted) degree $m$. We can then define the \emph{Green-Griffiths jet spaces} as $X_k^{GG} = \mathbf{Proj}_X(E_{k, \bullet}^{GG} \Omega_X)$. This projective bundle comes with its natural tautological (orbifold) line bundle $\mathcal O_k^{GG}(1)$.
\medskip

Recall that the algebra $E_{k, \bullet}^{GG} \Omega_X$ is endowed with a canonical $\mathbb N^k$-filtration, that we will denote by $F_\bullet E_{k, \bullet}^{GG} \Omega_X$, and which is compatible with the structure of $\mathcal O_X$-algebra. Its associated graded algebra is isomorphic to 
\begin{equation} \label{eqgradedterm}
\mathrm{Gr}_F \left( E_{k, \bullet}^{GG} \Omega_X \right) \cong S^{\bullet} \left( \Omega_X^{(1)} \oplus ... \oplus \Omega_X^{(k)} \right),
\end{equation}
where the right hand term is defined as in \eqref{eqdefsym}.
\medskip

To prove Theorem \ref{thmprinc}, the strategy coming from \cite{dem11} is to control the growth of $\chi^{[i]}(X, E_{k, m}^{GG} \Omega_X)$ when $X$ is a given manifold of general type. The next proposition shows that, to do so, it is sufficient to control the growth of $\chi^{[i]}$ on the graded term \eqref{eqgradedterm}.

\begin{prop} \label{propcompjetvb} For any $k, m \in \mathbb N$, and any $0 \leq i \leq n$, we have
\begin{equation} \label{equpperboundjet}
\chi^{[i]} \left( X, E_{k, m}^{GG} \Omega_X \right) \leq \chi^{[i]} \left( X, S^m (\Omega_X^{(1)} \oplus ... \oplus \Omega_X^{(k)}) \right).
\end{equation}
\end{prop}
\begin{proof}
The following simple argument has been communicated to me by L. Darondeau. The proposition comes from the fact that for any vector bundle $E \longrightarrow X$, and any filtration $F_\bullet E$, we have:
$$
\chi^{[i]} (E) \leq \chi^{[i]}(\mathrm{Gr}_F(E)).
$$
	To prove this last fact, we can reason by induction on the length of the filtration: the base case is when this length is equal to $2$. Then, the result boils down to showing that for any exact sequence $0 \longrightarrow F \longrightarrow E \longrightarrow G \longrightarrow 0$, we have $\chi^{[i]}(E) \leq \chi^{[i]}(F) + \chi^{[i]}(G)$. The last formula follows directly from the long exact sequence in cohomology, as in Step 4 of the proof of Theorem \ref{thmmorse}. 
\end{proof}

\begin{rem}

When $m$ is divisible enough, we can draw a more geometric picture for the last proposition, using the Rees deformation  (see \cite{BG96, cad17}). The latter yields a graded sheaf $\mathcal E$ of $\mathcal O_{X \times \mathbb{C}}$-algebras over $X \times \mathbb C$, whose specialization to the central fiber $X \times \left\{0\right\}$ is identified to  $S^{\bullet} \, ( \Omega_X^{(1)} \oplus ... \oplus \Omega_X^{(k)})$, and whose specialization to any other fiber $X \times \left\{ t \right\}$ ($t \neq 0$) identifies with $E_{k, \bullet}^{GG} \Omega_X$. Projectivizing $\mathcal E$ over $X \times \mathbb C$, we get a morphism of varieties $\mathcal P \longrightarrow X \times \mathbb C$. This can be seen as a family of projectivized bundles $(P_t \longrightarrow X)_{t \in \mathbb C}$, with natural identifications  $P_0 \cong \mathbb{P} (\Omega_X^{(1)} \oplus ... \oplus \Omega_X^{(k)})$ and $P_t \cong X_k^{GG}$ ($t\neq0$). There is also a natural orbifold line bundle $\mathcal L \longrightarrow \mathcal P$, which restricts to the respective tautological line bundles of $X_k^{GG}$ and $\mathbb{P} (\Omega_X^{(1)} \oplus ... \oplus \Omega_X^{(k)})$. If $m$ is divisible enough, the standard line bundle $\mathcal L^{\otimes m}$ is flat over $\mathbb C$ since $\dim \mathbb C = 1$ and $\mathcal P$ is reduced. Thus, we can apply the upper semi-continuity property of $\chi^{[i]}$  (see Demailly \cite{dem95}) to obtain, for a generic $t \neq 0$,
$$
\chi^{[i]} (P_t, \mathcal L_t ^{\otimes {m}}) \leq \chi^{[i]} (P_0, \mathcal L_0^{\otimes m}).
$$
Now, we have $\chi^{[i]} (P_t, \mathcal L_t^{\otimes m}) = \chi^{[i]} (X_k^{GG}, \mathcal O_k^{GG}(m)) = \chi^{[i]} (X, E_{k, m}^{GG} \Omega_X)$, and the right hand side identifies similarly with the right hand side of \eqref{equpperboundjet}. 
\end{rem}

\subsection{Relative version. Twists by ideal sheaves} \label{sectrel} In this section, we state a relative version of Proposition \ref{propcompjetvb}, where we tensor additionally the sheaves $\Omega_X$ by ideal sheaves of the form $\mathcal O(-E)$, for some auxiliary effective divisor.\medskip

Let $p : X' \longrightarrow X$ be a morphism of projective manifolds, and let $E$ be an effective divisor on $X'$. Then, for any $k, m \in \mathbb N$, the vector bundle $p^\ast E_{k, m}^{GG}$ admits a filtration with graded term
$$
p^\ast S^m (\Omega_X^{(1)} \oplus ... \oplus \Omega_X^{(k)}).
$$ 

This admits
\begin{equation} \label{eqsubsheaf}
	\mathcal O(-m E) \otimes p^\ast S^m (\Omega_X^{(1)} \oplus ... \oplus \Omega_X^{(k)})
\end{equation} 
as a subsheaf, and the $\mathcal O_{X'}$-module $\mathcal E_m = \mathcal O(-m E) \otimes p^\ast E_{k,m}^{GG} \Omega$ has a natural induced filtration for which the graded module is isomorphic to \eqref{eqsubsheaf}. 

In this context, the proof of Proposition \ref{propcompjetvb} applied to $\mathcal E_m$ yields the following result.

\begin{prop} \label{proprelative}
Let $p : X' \longrightarrow X$ be a surjective morphism of projective manifolds of dimension $n$, and let $E$ be an effective divisor on $X'$. Fix $k \in \mathbb N$. Then, for any $m \in \mathbb N$, there exists a subsheaf $\mathcal E_m \subseteq p^\ast E_{k, m}^{GG} \Omega_X$ such that, for any $i$ $(0 \leq i \leq n)$, we have
$$
	\chi^{[i]} \left( X', \mathcal E_m \right) \leq \chi^{[i]} \left( X',	\mathcal O(-m p^\ast E) \otimes p^\ast S^m (\Omega_X^{(1)} \oplus ... \oplus \Omega_X^{(k)}) \right).
$$
\end{prop}

In particular, specializing Proposition \ref{proprelative} to $i = 1$, and applying the two inequalities $h^0(p^\ast E_{k,m}^{GG} \Omega_X) \geq h^0(\mathcal E_m)$ and $h^0 \geq - \chi^{[1]}$, we obtain
\begin{prop} \label{proptwisteestimate}
Let $p : X' \longrightarrow X$ be a surjective morphism of projective manifolds of dimension $n$, and let $E$ be an effective divisor on $X'$. Then, for any $k, m \geq 1$, we have
$$
h^0(X', p^\ast E_{k,m}^{GG} \Omega) \geq - \chi^{[1]} \left( X',	\mathcal O(-m p^\ast E) \otimes p^\ast S^m (\Omega_X^{(1)} \oplus ... \oplus \Omega_X^{(k)}) \right)
$$
\end{prop}
\medskip

The point of Proposition \ref{propcompjetvb} (or of the more general Proposition \ref{proptwisteestimate}) is to permit us to limit our study to the symmetric algebra ${S^\bullet (\Omega^{(1)} \oplus ... \oplus \Omega^{(k)})}$. In the following sections, we will perform this study for the more general case of symmetric products of the form $S^m(E^{(1)} \oplus ... \oplus E^{(k)})$, where $E$ is any vector bundle with $\mathrm{det} \, E$ big. The goal of the next section is to show that it is actually enough to study the case where $E$ is a direct sum of line bundles.

\subsection{Reduction to a sum of line bundles} \label{sectredsum}

The main goal of the present section is to prove the following \emph{splitting principle}, which will drastically simplify our application of Morse inequalities to weighted direct sums. 
\begin{prop} \label{propredlb}
Let $\mathbb E = E_1^{(a_1)} \oplus ...\oplus E_r^{(a_r)}$ be a weighted direct sum over a complex projective manifold $X$. Then, there exists a smooth modification $p : \widetilde{X} \longrightarrow X$, and a weighted direct sum $\overline{\mathbb E} = \overline{E_1}^{(a_1)} \oplus ... \oplus \overline{E_r}^{(a_r)}$ over $\widetilde{X}$, such that
\begin{enumerate} \itemsep=0em
\item each $\overline{E_i}$ is a direct sum of line bundles over $\widetilde{X}$, of the same rank as $E_i$ ;
\item for each $i$, $\det(\overline{E_i}) \cong p^\ast \, \det (E_i)$,
\item for any $i$ (with $1 \leq i \leq \dim X$), we have the asymptotic inequality
$$
\chi^{[i]} (X, S^m \, \mathbb E) \leq \chi^{[i]} (\widetilde{X}, S^m \, \overline{\mathbb E}) + O(m^{n + {\rm rk}\, \mathbb E - 1}).
$$
\end{enumerate}
\end{prop}
\begin{rem}
	The item (2) implies in particular that $\det(\overline{E_i})$ is big if $\det(\overline{E_i})$ is.
\end{rem}

We begin by a simple lemma.

\begin{lem} \label{lemlb}
Let $X$ be a projective complex manifold, and let $E \longrightarrow X$ be a vector bundle of rank $m$.  There exists a smooth modification $p : \widetilde{X} \longrightarrow X$ such that $p^\ast E$ admits a total filtration by vector subbundles:
$$
F_1 \subsetneq F_2 \subsetneq ... \subsetneq F_m = p^\ast E.
$$
\end{lem}

\begin{proof}
We proceed by induction on $m = \mathrm{rk}\, E$. The case where $\mathrm{rk} E = 1$ is trivial. Let us assume that $\mathrm{rk}\, E \geq 2$.

Let $U \subseteq X$ be a Zariski open subset such that $E|_U$ is trivialized by a frame $(e_1, ..., e_m)$. Let $\mathcal F_0 \subseteq E$ be the unique saturated subsheaf such that $\mathcal F_0|_{U} = \left< e_1, ..., e_{m-1} \right>$. Then, we have an exact sequence
$$
0 \longrightarrow \mathcal F_0 \longrightarrow E \longrightarrow \mathcal L_0 \longrightarrow 0,
$$
	where $\mathcal L_0$ is a torsion free rank one coherent sheaf. By e.g. \cite[Chapter V, \S 5, 6]{kob87}, $L_0 = \mathcal L_0^{\vee \vee}$ is a line bundle over $X$, and we have a natural identification $\mathcal L_0 = \mathcal I \otimes_{\mathcal O_X} L_0$ for some ideal sheaf $\mathcal I$ over $X$. Now, use Hironaka's principalization theorem to obtain a smooth modification $\pi : X' \longrightarrow X$ such that $\pi^\ast \mathcal I = \mathcal O(-E)$ as a subsheaf of $\mathcal O_{X'}$, for some effective Cartier divisor $E$. Let $L = \mathcal O(-E) \otimes_{\mathcal O_{X_0}} \pi^\ast L_0$: this sheaf is a line bundle, and the morphism $E \twoheadrightarrow \mathcal L_0 = \mathcal I {\otimes} L_0$ induces a natural surjection $\pi^\ast E \twoheadrightarrow L$. This map is a surjection of locally free sheaves, hence its kernel is a subvector bundle $F \subseteq \pi^\ast E$, and $\quotientd{\pi^\ast E}{F} = L$. Since $\mathrm{rk} \, F < \mathrm{rk} \, E$, we can now apply the induction hypothesis to $F$ on the manifold $X'$, and get a modification $\widetilde{X} \longrightarrow X$ dominating $X'$, and satisfying our requirements.
\end{proof}

We have then reduced to the case where the vector bundles are filtered by complete flags of vector subbundles. Then, further reducing to direct sums of line bundles is not hard.

\begin{lem} \label{lemdeffilt}
Let $E_1, ..., E_m$ be vector bundles over a projective manifold $X$, and let $a_1, ..., a_m \in \mathbb N^\ast$. Assume that each $E_i$ admits a filtration by vector subbundles
\begin{equation} \label{eqfiltration}
E_{i, 1} \subsetneq ... \subsetneq E_{i, r_i} = E_i,
\end{equation}
where $r_i = \mathrm{rk}\, E_i$. Let $L_{i, 1}, ..., L_{i, r_i}$ be the successive line bundle quotients. Denote by $\overline{E_i} = L_{i, 1} \oplus ... \oplus L_{i, r_i}$ the graded bundle of the filtration \eqref{eqfiltration}. Then, for any $i$ $(1 \leq i \leq n)$, and any $m \geq 1$, we have
\begin{equation} \label{equpperboundlb}
\chi^{[i]} (S^m (E_1 ^{(a_1)} \oplus ... \oplus E_k^{(a_k)})) \leq \chi^{[i]} (S^m (\overline{E_1}^{(a_1)} \oplus ... \oplus \overline{E_k}^{(a_k)})).
\end{equation}
\end{lem}

\begin{proof}
We can use the same argument as in the proof of Proposition \ref{propcompjetvb}.
Let $\mathbb{E} = E_1^{(a_1)} \oplus ... \oplus E_r^{(a_r)}$, and $\overline{\mathbb E} = \overline{E_1}^{(a_1)} \oplus ... \oplus \overline{E_r}^{(a_r)}$. The filtrations given in \eqref{eqfiltration} induce a natural filtration on $S^\bullet \mathbb E$, with graded algebra equal to $S^\bullet \overline{\mathbb E}$. Now, we use again the comparison result for $\chi^{[i]}$ applied to a filtered vector bundle and to its graded terms, to get
$$
\chi^{[i]} (X, S^m \mathbb E) \leq \chi^{[i]}(X, S^m \overline{\mathbb E}).
$$
This completes the proof.
\end{proof}

\begin{proof}[Proof of Proposition \ref{propredlb}]
By Lemma \ref{lemlb}, there exists a modification $p : \widetilde{X} \longrightarrow X$ such that all the direct summands of $p^\ast \mathbb E = p^\ast E_1^{(a_1)} \oplus ... \oplus p^\ast E_r^{(a_r)}$ admit filtrations with direct sums of line bundles as graded terms. The same Leray's spectral sequence argument as in Lemma \ref{lemmodification} yields
$$
\chi^{[i]} (X, S^m \, \mathbb E) \leq \chi^{[i]} (\widetilde{X}, S^m \, p^\ast \mathbb E) + O(m^{n + \mathrm{rk} \mathbb E - 1}).
$$
To conclude, it suffices to apply Lemma~\ref{lemdeffilt} to the weighted direct sum $p^\ast \mathbb E$ on $\widetilde{X}$. For each $i$ ($1 \leq i \leq r$), it yields a direct sum of vector bundles $\overline{E_i}$, with $\det \overline{E_i} = \det (p^\ast E_i)$, and such that, for any $m$,
$$
\chi^{[i]} (\widetilde X,  S^m p^\ast \mathbb E) \leq  \chi^{[i]} (\widetilde X,  S^m  \overline{\mathbb E}),
$$
where $\overline{\mathbb E} = \overline{E_1}^{(a_1)} \oplus ... \oplus \overline{E_r}^{(a_r)}$. This ends the proof.
\end{proof}

\section{Morse inequalities for symmetric products of weighted direct sums}

\subsection{Statement of the result} \label{sectstatement}
The main theorem of this section is an asymptotic estimate on the growth of quantities of the form $\chi^{[i]} (X, N^{\otimes m} \otimes S^m\, \mathbb E)$, where $\mathbb E$ is a \emph{weighted direct sum} of line bundles, and $N$ is an arbitrary $\mathbb Q$-line bundle; it will be the central step in our proof of Theorem \ref{thmprinc}.

To simplify a bit the exposition, we have chosen to state first a result holding for standard line bundles ; we will present its generalization to $\mathbb Q$-line bundles later on. Before stating this first result, we need to introduce a few notations.

\medskip

\begin{defi} \label{defisimplex}
	Let $m \in \mathbb N$. A \emph{$m$-dimensional simplex} $\Delta$ is a metric space isomorphic to the convex envelop of $m+1$ points in $\mathbb R^m$, such that each $p$ of them generate an affine $(p-1)$-space. We will sometimes write $\Delta \; \accentset{\circ}{\subset} \; \mathbb R^m$ to emphasize the fact that $\Delta$ has non-empty interior in $\mathbb R^m$ (or equivalently, that $\dim \Delta = m$) and to oppose this situation to the case of a $(m-1)$-dimensional simplex included in $\mathbb R^m$.

	For all $(a_1, ..., a_r) \in \mathbb N_{\geq 1}^r$, we define the $(r-1)$-dimensional simplex $$\Delta_{(a_1, ..., a_r)} = \{ (t_1, .., t_r) \in \mathbb R^r \; | \; \sum_i a_i t_i = 1 \} \subseteq \mathbb R^r.$$ For any $m \in \mathbb N$, we will simply denote by $\Delta^m$ the $m$-dimensional simplex $\Delta_{(1, ..., 1)}$ ($1$ repeated $m+1$ times).
\end{defi}

Consider now a complex variety $X$ of dimension $n$, and let $L_1, ..., L_r$ be line bundles over $X$. Let $\Sigma$ be a stratification adapted to all $L_1, ..., L_r$, and for each $i \in \llbracket 1, r \rrbracket$, choose a trivialization $\mathbf e_i$ of $L_i$ over $\Sigma$. Let $\mathcal T$ be the tree associated to $\Sigma$. For each edge $\mathfrak{s}$ in $\mathcal T$, denote by $m_{i}^\mathfrak{s}$ the marking of $\mathfrak{s}$ associated to the trivialization $\mathbf e_i$.  For all $i \in \llbracket 1, n \rrbracket$, we define a piecewise polynomial function on $\mathbb R^r$, as follows.

\begin{defi} \label{defiupsilon} Let $(t_1, ..., t_r) \in \mathbb R^r$. Mark each edge $\mathfrak{s}$ in $\mathcal T$ with the real number $t_1 m_{1}^{\mathfrak{s}} + ... + t_r m_{r}^\mathfrak{s}$. For each complete path $\sigma$ in $\mathcal T$, denote by $C_\sigma$ the product of all markings along the edges of $\sigma$. Then, for all $i \in \llbracket 1, n \rrbracket$, we let
$$
\upsilon_{[\leq i]} (t_1, ..., t_r) = \sum_{\mathrm{index}(\sigma) \leq i} C_{\sigma},
$$
where the sum runs among all the complete paths of index $\leq i$ in $\mathcal T$, i.e. among the paths with less than $i$ negative markings.
\end{defi}

It is easy to check that the functions $\upsilon_{[\leq i]}$ are piecewise polynomial, and homogeneous of degree $n$, i.e. $\upsilon_{[\leq i]} (\lambda \cdot u) = \lambda^n \; \upsilon_{[\leq i]} (u)$ for $\lambda \in \mathbb R_+$. We can now state the following theorem.

\begin{thm} \label{thmineqintegral} Let $\underline{a} = (a_1, ..., a_r) \in \mathbb N_{\geq 1}^r$, and consider the $(r-1)$-dimensional simplex $\Delta_{\underline{a}}$, according to Definition \ref{defisimplex}. Then, for all $i \in \llbracket 0, n \rrbracket$, we have the following asymptotic upper bound.
\begin{align*}
	\chi^{[i]} (X, S^m (&L_1^{(a_1)} \oplus ...  \oplus L_r^{(a_r)}) )  \\
& \leq  \frac{\mathrm{gcd}(a_1, ..., a_r)}{a_1 ... a_r} \binom{n + r - 1}{r - 1} \left[ \int_{\Delta_{\underline{a}}} \upsilon_{[\leq i]} dP \right] \frac{m^{n+ r - 1}}{(n+ r - 1)!} + o(m^{n+r - 1}),
\end{align*}
where $P$ is the uniform probability measure on the simplex $\Delta_{\underline{a}}$, i.e. the unique probability measure which is the restriction of a translation invariant measure on $\mathbb R^{r-1}$.
\end{thm}

The proof of Theorem \ref{thmineqintegral} will proceed in several steps. We will first give an extension of Theorem \ref{thmmorse} to a context where several line bundles are introduced. Then, we will apply this estimate to give an upper bound on the $\chi^{[i]} (X, L_1^{\otimes l_1} \otimes ... \otimes L_r^{\otimes l_r})$, where $(l_1, ..., l_r)$ belongs to a small angular sector of $\mathbb N^r$. Finally, we will sum over a covering of $\mathbb N^r$ by such arbitrarily small angular sectors, to get a Riemann sum leading to the integral appearing in Theorem \ref{thmineqintegral}.

\subsection{Morse inequalities for several line bundles. Direct sums of line bundles}

Let $X$ still denote a complex variety, and let $\mathcal L$ be a finite set of line bundles on $X$. Assume $\Sigma$ is a stratification on $X$, adapted to all $L \in \mathcal L$. For each $L \in \mathcal L$, we let $\mathbf e_L$ be a stratification of $L$ on $\Sigma$, and we let $\underline{\Sigma} = ( \Sigma, (\mathbf{e}_L)_{L \in \mathcal L} )$. The next definition will introduce a quantity that will serve as leading coefficient for an upper bound on the numbers $\chi^{[i]} (X, L_1^{\otimes a_1} \otimes  ... L_r^{\otimes a_r})$ as $a_1, ..., a_r \longrightarrow + \infty$, where $L_1, ..., L_r \in \mathcal L$ are arbitrary elements.
\medskip

Let $\mathcal T$ be the tree associated to $\Sigma$, and let $\mathcal E$ be the set of its edges. For all $\mathfrak{s} \in \mathcal E$, we let $m_{L}^{\mathfrak{s}}$ be the marking of $e$ provided by the trivialization $\mathbf{e}_L$. Now, for any mapping $\phi : \mathcal E \longrightarrow \mathcal L$, we let $\mathcal T_{\phi}$ be the \emph{marked} tree where each edge $\mathfrak{s}$ is labeled by the marking $m_{\phi(\mathfrak{s})}^\mathfrak{s}$.

\begin{defi} \label{defileadingcoeff} Let $l \in \llbracket 0, n \rrbracket$. For all $\phi : \mathcal E \longrightarrow \mathcal L$, and any complete path $\sigma$ in $\mathcal T$, we let $C_{\sigma, \phi}$ be the product of the labels along the edges of $\sigma$ in $\mathcal T_{\phi}$. We let
$$
c(\phi, \underline{\Sigma})_{[ \leq i]} = \sum_{\mathrm{index} (\sigma) \leq  i} C_{\sigma, \phi},
$$
where the sum runs among all complete paths in $\mathcal T_{\phi}$ with less than $i$ negative labels. 

We let
$$
(-1)^i c(\mathcal L, \underline{\Sigma})_{[\leq l]} = \underset{\phi : \mathcal E \longrightarrow \mathcal L}{\max} (-1)^i c(\phi, \underline{\Sigma})_{[\leq l]}.
$$
\end{defi}

The previous definition mimicks Proposition \ref{propgraph} to provide the leading coefficient in our version of Morse inequalities with several line bundles. Before stating this result, let us introduce the following simplifying notation.
\medskip

\begin{nota} Let $\underline{l} = (l_L)_{L \in \mathcal L}$ be a family of integers. We write $\mathcal L^{\otimes \underline{l}} = \bigotimes_{L \in \mathcal L} L^{\otimes l_L}$. For all $L \in \mathcal L$, we write $\underline{\delta_L} = (\delta_{L, L'})_{L' \in \mathcal L}$, where $\delta_{L, L'} = 1$ if $L = L'$ and $\delta_{L, L'} = 0$ otherwise.
\end{nota}

\begin{prop} \label{propvariantmorse} With the notations above, let $M$ be any line bundle on $X$.

Then for any $i \in \llbracket 0, n \rrbracket$, any $m \in \mathbb N$, and any $\underline{l} = (l_L)_{L \in \mathcal L}$ such that $\sum_L l_L = m$, we have
$$
\chi^{[i]} (X, \mathcal L^{\otimes \underline{l}} \otimes M) \leq (-1)^i c(\mathcal L, \underline{\Sigma})_{[\leq j]} \frac{m^n}{n!} + O(m^{n-1}).
$$ 
\end{prop}
\begin{proof}
	\emph{Step 1. Initialization of the induction.} The proof is very close to the one of Theorem \ref{thmmorse}. We reason by induction on $\dim X$. When $\dim X = 0$, the result is trivial, with $O(m^{-1}) = 0$.

Let us assume that the result has been proved for all dimensions up to $n - 1$. Write $\Sigma \equiv X_0 \overset{f_0}{\longrightarrow} ... \longrightarrow X_{n-1} \overset{f_{n-1}}{\longrightarrow} X_n = X$, and for each $L \in \mathcal L$, let $s_L$ be the trivialization of $L$ over the open affine subset $U_n = X_n \setminus f_{n-1}(X_{n-1})$ provided by $\mathbf{e}_L$.
\medskip

	We can now argue as in the proof of Theorem \ref{thmmorse}: it suffices to prove the seemingly more general proposition.
\begin{prop} \label{propredmorsegen} Under the hypotheses of Theorem \ref{propvariantmorse}, let $p: X' \longrightarrow X$ be a proper birational morphism, and let $M'$ be any line bundle on $X'$. Then, for each $i$, and any $m$ divisible by $d$, we have
	Then for any $i \in \llbracket 0, n \rrbracket$, any $m \in \mathbb N$, and any $\underline{l} = (l_L)_{L \in \mathcal L}$ such that $\sum_L l_L = m$, we have
$$
\chi^{[i]} (X, p^\ast \mathcal L^{\otimes \underline{l}} \otimes M') \leq (-1)^i c(\mathcal L, \underline{\Sigma})_{[\leq j]} \frac{m^n}{n!} + O(m^{n-1}).
$$ 
\end{prop}

	Thus, it suffices to deal with the case where $X$ is smooth, and $D_n = f_{n-1}(X_{n-1})$ has simple normal crossing support. For simplicity of notations, we will conclude the proof in the case where $D_n$ is even \emph{irreducible} (hence smooth). For all $L \in \mathcal L$, we let $m_L$ be the multiplicity of $s_L$ along the unique component of $D_n$. 
\medskip

	\emph{Step 2. We bound the difference between terms given by two close $r$-uples.} Let $\underline{l} = (l_L)_{L \in \mathcal L}$ be such that $\sum_{L \in \mathcal L} l_L = m$, with $m \geq 1$. Then, for some $L \in \mathcal L$, we have $\underline{l} - \underline{\delta_L} \in \mathbb N^\mathcal L$. The same arguments as in Step 4 in the proof of Theorem \ref{thmmorse} yield
\begin{align*}
\chi^{[i]} (X, \mathcal L^{ \otimes \underline{l}} \otimes M)  - \chi^{[i]} &  (X, \mathcal L^{ \otimes (\underline{l} - \underline{\delta_L})} \otimes M) 
\\  \leq  \mathbbm 1_{\{m_L > 0\}} & \; m_L \; \chi^{[i]} (D_n, \mathcal L^{\otimes \underline{l}} \otimes M) \; + \; \mathbbm 1_{\{m_L < 0\}} \; (-m_L) \; \chi^{[i - 1]} (D_n, \mathcal L^{\otimes (\underline{l} - \underline{\delta_L})} \otimes M) \\
 & + O(m^{n-2}).
\end{align*}

Let $\Sigma'$ be stratification induced on $D_{n}$ by $\Sigma$, and let $\mathbf e'_{N}$ be the trivializations of the $N \in \mathcal L$ induced by the $\mathbf e_N$ on $\Sigma'$. Writing $\underline{\Sigma'} = (\Sigma', (\mathbf e'_N)_{N \in \mathcal L})$, we can apply the induction hypothesis to obtain
\begin{align} \nonumber \chi^{[i]} (X, \mathcal L^{ \otimes \underline{l}} \otimes M)  - \chi^{[i]} &  (X, \mathcal L^{ \otimes (\underline{l} - \underline{\delta_L})} \otimes M) 
\\ \label{eqinductseveral}  \leq  (-1)^i \mathbbm 1_{\{m_L > 0\}} & \; m_L \; c(\mathcal L, \underline{\Sigma'})_{[\leq i]} \frac{m^{n-1}}{(n-1)!}  \; + \; (-1)^i \mathbbm 1_{\{m_L < 0\}} \; m_L \; c(\mathcal L, \underline{\Sigma'})_{[\leq i-1]} \frac{(m-1)^{n-1}}{(n-1)!} \\ \nonumber
 & + O(m^{n-2}).
\end{align}

\begin{lem} Let $\mathfrak{s}$ be the edge linking the unique component of $X_{n-1}$ to $X_n$ in $\mathcal T$. Then
\begin{align} \nonumber 
(-1)^i ( \mathbbm 1_{\{m_L > 0\}} \; m_L \; c(\mathcal L, \underline{\Sigma'})_{[\leq i]}  + \mathbbm 1_{\{m_L < 0\}} & \; m_L \; c(\mathcal L, \underline{\Sigma'})_{[\leq i - 1]})  \\ \label{eqinductphi}
& = \underset{\phi : \mathcal E \longrightarrow \mathcal L \; | \; \phi(\mathfrak{s}) = L }{\max} (-1)^i c(\phi, \underline{\Sigma})_{[\leq i]},
\end{align}
where the maximum in the right hand side runs among all $\phi$ such that $\phi(\mathfrak{s}) = L$. In particular, by Definition \ref{defileadingcoeff}, the left hand side is bounded from above by $(-1)^i c(\mathcal L, \underline{\Sigma})_{[\leq i]}$.
\end{lem}
 
The lemma comes right away from Definition \ref{defileadingcoeff}: it suffices to evaluate the maximum on the right hand side, by distinguishing among the two possible signs of $m_L$. Now, since $(m-1)^{n-1} = m^{n-1} + O(m^{n-2})$, inserting \eqref{eqinductphi} in \eqref{eqinductseveral} gives
$$
\chi^{[i]} (X, \mathcal L^{ \otimes \underline{l}} \otimes M)  - \chi^{[i]}  (X, \mathcal L^{ \otimes (\underline{l} - \underline{\delta_L})} \otimes M) \leq (-1)^i c(\mathcal L, \underline{\Sigma})_{[\leq i]} \frac{m^{n-1}}{(n-1)!} + O(m^{n-2}).
$$

	\emph{Step 3. We sum these differences over a path leading to a given $\underline{l}$.} Consider now a sequence $\underline{l}_1, ..., \underline{l}_m = \underline{l}$, where for all $j$, we have $\underline{l}_j = \underline{l}_{j-1} + \underline{\delta_{L_j}}$ for some $L_j \in \mathcal L$. For all $j$, we have $\sum_{L} (l_j)_{L} = j$, so we can sum the last inequality for all $\underline{l_j}$, to get
$$
\chi^{[i]}(X, \mathcal L^{\otimes \underline{l}}) \leq (-1)^i c(\mathcal L, \underline{\Sigma})_{[\leq i]} \sum_{j = 1}^m \frac{j^{n-1}}{(n-1)!} + O(m^{n-1}).
$$
This gives the result, since $\sum_{1 \leq j \leq m} \frac{j^{n-1}}{(n-1)!} = \frac{m^n}{n!} + O(m^{n-1})$.
\end{proof}

\begin{rem} \label{remroot}
	If we replace Proposition \ref{propredmorsegen} by the corresponding statement where $p$ is a generically finite morphism (see Proposition \ref{propredmorse}), and if we use the method of Step 3 in the proof of Theorem \ref{thmmorse}, it is possible to adapt the previous proof to the case where the $L \in \mathcal L$ are standard line bundles endowed with fractional trivializations $\frac{1}{d_L} \mathbf{e}_L$. The definition of $c_1(\mathcal L, \underline{\Sigma})_{[\leq j]}$ is obtained by replacing the markings $m_L$ by $\frac{m_l}{d_L}$.
\end{rem}

Let us mention yet another variant, which will give a better estimate in Proposition \ref{propvariantmorse} when $\underline{l}$ belongs to a narrow angular sector of $\mathbb N^{\mathcal L}$.
\medskip

We assume that $\mathcal L = \{ L_1, ..., L_r \}$, and we pick $\underline{u_1}, ..., \underline{u_p} \in \mathbb N^r$. For each $i$, we write $\underline{u_i} = (u_{i,1}, ..., u_{i, r})$, and we let $M_i = \mathcal L^{\otimes \underline{u_i}}$. Each $M_i$ is endowed with a natural trivialization $\mathbf{f}_i$ on $\Sigma$, given on a particular stratum by
\begin{equation} \label{eqdefifi}
f_i = (e_1)^{\otimes u_{i, 1}} \otimes ... \otimes (e_r)^{\otimes u_{i, r}},
\end{equation}
where the $e_j$ are the trivializations of the $L_j$ on the specified stratum. Let $\mathcal M = \{M_1, ..., M_p \}$, and $\underline{\Sigma}^{\mathcal M} = (\Sigma, (\mathbf f_j)_j)$.

Using Definition \ref{defileadingcoeff} with $\mathcal L$ replaced by $\mathcal M$, we can define the quantities $c(\mathcal M, \underline{\Sigma}^{\mathcal M})_{[\leq i]}$. It is clear by construction that the latter are continuous piecewise polynomial functions in the $u_{i}$. More precisely, we have the following.

\begin{lem} \label{lemcontfunction} There exists a continuous piecewise polynomial function $\varphi_{[\leq j]} : (\mathbb R^r_+)^p \longrightarrow \mathbb R$, homogeneous of degree $n$, such that for all $\underline{u_1}, ..., \underline{u_p} \in \mathbb N^r$ as above, we have
$$
c(\mathcal M, \underline{\Sigma}^{\mathcal M})_{[\leq i]} = \varphi_{[\leq i]} (\underline{u_1}, ..., \underline{u_p}).
$$
Moreover, $\varphi_{[\leq j]}$ depends only on $p$ and on the data of $\Sigma, \mathbf{e}_1, ..., \mathbf{e}_p$.
\end{lem}
\begin{proof}
By Definition \ref{defileadingcoeff}, the real number $c(\mathcal M,  \underline{\Sigma}^{\mathcal M})_{[\leq i]}$ is a sum of maxima and minima of $n$-homogeneous piecewise polynomials functions in the markings of the trees $\mathcal T_{\phi}$. Each one of these markings being a linear form in $(\underline{u_1}, ..., \underline{u_r}) \in (\mathbb R_+^r)^p$ by definition of the $\mathbf f_i$ (see \eqref{eqdefifi}), we get the result.
\end{proof}

Comparing Definitions \ref{defiupsilon} and \ref{defileadingcoeff} with Lemma \ref{lemcontfunction} gives the following simple relation between $\varphi_{[\leq i]}$ and $\upsilon_{[\leq i]}$:

\begin{lem} For any $t \in \Delta_{\underline{a}}$, we have
$$
\upsilon_{[\leq j]} (t) = \varphi_{[\leq j]} (t, ..., t).
$$
\end{lem}

We can now give the following refinement of Proposition \ref{propvariantmorse} when $\underline{l}$ belongs to a narrow cone of the form $\sum_j \mathbb R_+ \cdot \underline{u_j}$, for $\underline{u_1}, ..., \underline{u_p} \in \mathbb N^r$.

\begin{prop} \label{propcone}
	Let $\underline{a} = (a_1, ..., a_r) \in \mathbb N_{\geq 1}^r$ and $\underline{u_1}, ..., \underline{u_p} \in \mathbb N^r$ be as before. For any $m \in \mathbb N$, we let $H_m = \{ \underline{l} = (l_j) \in \mathbb N^r \; | \; \sum_j a_j l_j = m \}$. Assume that there exists $t \in \Delta_{\underline{a}}$ and $\lambda, \epsilon > 0$ such that for all $j \in \llbracket 1, p \rrbracket$, we have $\norm{t - \frac{1}{\lambda}\underline{u_j}}_{\infty} < \epsilon$. Then, for any $i \in \llbracket 0, n \rrbracket$, any $m \in \mathbb N$, and any $\underline{l} \in (\sum_j \mathbb R_+ \cdot \underline{u_j}) \cap H_m$,
\begin{equation} \label{eqsmallcone}
\chi^{[i]} (X, \mathcal L^{\otimes \underline{l}} \otimes M) \leq (-1)^i (\varphi_{[\leq j]}(t) + O(\epsilon) ) \frac{m^n}{n!} + O(m^{n-1}),
\end{equation}
where the constant appearing in the $O(\epsilon)$ term depends only on $\Sigma, \mathbf{e}_1, ..., \mathbf{e}_r$ and the $a_j$.
\end{prop}
\begin{proof}

\emph{Step 1. We determine a first asymptotic expansion for the right hand side of \eqref{eqsmallcone}}. In this step, the $\underline{u_j}$ are \emph{fixed}, and only $\underline{l}$ and $m$ are allowed to vary.
\medskip

Since $\underline{l} \in (\sum_j \mathbb R_+ \cdot \underline{u_j}) \cap \mathbb N^r$, we can write $\underline{l} = \underline{d} + \underline{v_0}$, with $\underline{d} \in \sum_j \mathbb N \cdot \underline{u_j}$ and $\underline{v_0} \in \left( \sum_j [0, 1] \cdot \underline{u_j}\right) \cap \mathbb N^r$. In particular $\norm{\underline{v_0}}_{\infty}$ is bounded by a constant independent of $m$. Write $\underline{d} = q_1 \, \underline{u_1} + ... + q_p \, \underline{u_p}$ and let $\underline{q} = (q_1, ..., q_p)$. We have then
\begin{align*}
\mathcal L^{\underline{l}} \otimes M & = M_1^{\otimes q_1} \otimes ... \otimes M_p^{\otimes q_p} \otimes \mathcal L^{\otimes \underline{v_0}} \otimes M \\
&  = \mathcal M^{\otimes \underline{q}} \otimes ( \mathcal L^{\otimes \underline{v_0}} \otimes M) 
\end{align*}
(recall that $M_j = \mathcal L^{\otimes \underline{u_j}}$). Since $\underline{v_0}$ is bounded, there is only a finite number of possible $\mathcal L^{\otimes \underline{v_0}} \otimes M$ that can appear in the previous equation. Thus, we can apply Proposition \ref{propvariantmorse} applied with $\mathcal L$ (resp. $\underline{\Sigma}$, resp. $M$) replaced by $\mathcal M$ (resp. $\underline{\Sigma}^{\mathcal M}$, resp. $\mathcal L^{\otimes \underline{v_0}} \otimes M$), to obtain
$$
\chi^{[i]} (X, \mathcal L^{\otimes \underline{l}} \otimes M) \leq (-1)^i c(\mathcal M, \underline{\Sigma}^{\mathcal M})_{[\leq i]} \frac{(\sum_j q_j)^n}{n!} + O((\sum_j q_j)^{n-1}).
$$
Lemma \ref{lemcontfunction} yields in turn:

\begin{equation} \label{eqasymptoticstep1}
\chi^{[i]} (X, \mathcal L^{\otimes \underline{l}} \otimes M) \leq (-1)^i \varphi_{[\leq i]} (\underline{u_1}, ..., \underline{u_p}) \frac{(\sum_j q_j)^n}{n!} + O((\sum_j q_j)^{n-1}).
\end{equation}

\emph{Step 2. Keeping the $\underline{u_j}$ fixed, we give an asymptotic expansion of the upper bound \eqref{eqasymptoticstep1} in terms of $m$.} A direct computation shows that $\sum_j q_j \leq \frac{1}{\max \norm{\underline{u_j}}_{\infty}} \frac{1}{\min_k a_k} ( \sum_j a_j l_j) = O(m)$, hence $O((\sum_j q_j)^{n-1}) = O(m^{n-1})$.

Moreover, we have
\begin{align*}
m = \sum_j a_j l_j & = \sum_{j} a_j d_j + \sum_{j} a_j (v_0)_j \\
	& = \sum_{k} q_k (\sum_j a_j u_{k, j} ) + \sum_{j} a_j (v_0)_j \\
	& = \sum_{k} q_k (\sum_j a_j \lambda (t_j + r_{k, j})) + O(1)
\end{align*}
where we let $r_{k, j} = \frac{1}{\lambda} u_{k, j} - t_j$. Note that the $O(1)$ term may depend on the $\underline{u_j}$ and the $a_j$, but not on $m$. Also, we have $|r_{k,j}| \leq \epsilon$ by hypothesis. Thus, since $\sum_{j} a_j t_j = 1$, still by hypothesis, we have
$$
m = (\sum_k q_k) \, \lambda \, (1 + O(\epsilon)) + O(1),
$$
where the constant appearing in the $O(\epsilon)$ term may depend on the $a_j$, but not on $m$ nor on the $u_j$, and the constant $O(1)$ may not depend on $m$.

Inserting this in \eqref{eqasymptoticstep1}, we get
$$
\chi^{[i]} (X, \mathcal L^{\otimes \underline{l}} \otimes M) \leq (-1)^i \varphi_{[\leq i]} (\underline{u_1}, ..., \underline{u_p}) \frac{1 + O(\epsilon)}{\lambda^n} \frac{m^n}{n!} + O(m^{n-1}).
$$
\emph{Step 3. We show that the leading coefficient obtained at Step 2 is close to $(-1)^i \upsilon_{[\leq j]} (u)$.} By homogeneity, we have $\frac{1}{\lambda^n} \varphi_{[\leq i]} (\underline{u_1}, ..., \underline{u_p}) = \varphi_{[\leq i]} (\frac{1}{\lambda} \underline{u_1}, ..., \frac{1}{\lambda} \underline{u_p})$. Since the definition of the function $\varphi_{[\leq j]}$ depends only on $\Sigma, \mathbf{e}_1, ..., \mathbf{e}_r$, and since this function is uniformly continuous in a compact neighborhood of $\Delta_{\underline{a}}$, we have finally
$$
\varphi_{[\leq j]} (\frac{1}{\lambda} \, \underline{u_1}, ..., \frac{1}{\lambda} \, \underline{u_p}) = \varphi_{[\leq j]} (t, ..., t) + O(\epsilon),
$$ 
where the constant appearing in $O(\epsilon)$ may only depend on $\Sigma, \mathbf{e}_1, ..., \mathbf{e}_r$ and the $a_j$. Since $\upsilon_{[\leq j]} (t) = \varphi_{[\leq j]} (t, ..., t)$, this ends the proof.
\end{proof}

The line bundle $\mathcal L^{\otimes \underline{l}}$ is one of the many line bundles appearing in the natural decomposition of the symmetric product $S^m (L_1^{(a_1)} \oplus ... \oplus L_r^{(a_r)})$. To prove Theorem \ref{thmineqintegral}, i.e. to obtain an upper bound on $\chi^{[i]} (X, S^m (L_1^{(a_1)} \otimes ... \otimes L_r^{(a_r)}))$, we will cover $\mathbb N^r$ by narrow cones of the form $\sum_j \mathbb R_+ \cdot \underline{u_j}$, and then apply inequality \eqref{eqsmallcone} to every line bundle appearing in the decomposition. Summing over all the cones, and then letting their width tend to $0$, will yield the result.

\begin{proof}[Proof of Theorem \ref{thmineqintegral}]

Let $v_1, ..., v_{n-1}$ be a basis of the primitive sublattice 
$$
H = \{(z_1, ..., z_r) \in \mathbb Z^r \; | \; \sum_i a_i z_i = 0 \} \subseteq \mathbb Z^r.
$$
For $m \in \mathbb N$, let $\mathcal H_m = \{ (t_i) \in \mathbb R^r \; | \; \sum_i a_i t_i = m \}$: with the notations of Proposition \ref{propcone}, we have then $H_m = \mathcal H_m \cap \mathbb N^r$. Let $\epsilon > 0$. Let $m_0 \in \mathbb N$ be such that $m_0 > \frac{\max_i \norm{v_i}_{\infty}}{\epsilon}$. 
\medskip

\emph{Step 1. We construct a partition of $m_0 \cdot \Delta_{\underline{a}}$ in elementary polyhedral cells.}
\medskip
 
For all $u \in H_{m_0}$, we let $C^\circ_u = u + \sum_{j=0}^{r-1} [0, 1[ \cdot v_j \subseteq H_{m_0}$, and $C_u = C^\circ_u \cap m_0 \cdot \Delta_{\underline{a}}$. The following facts are easy to check.

\begin{lem}  \label{lemauxiliary}
\begin{enumerate}
\item Each $C_u$ $(u \in H_{m_0})$ is a rational polyhedron contained in $m_0 \cdot \Delta_{\underline{a}}$. We have $\cup_{u \in H_{m_0}} C_u = m_0 \cdot \Delta_{\underline{a}}$ and $C_u \cap C_{u'} = \emptyset$ for $u \neq u'$. 
\item We have, for any fixed $m_0$, and any $m \geq m_0$
$$
\mathrm{card} \left((\mathbb R_+ \cdot C_u) \; \cap \; H_m \right) = O \left( \left(\frac{m}{m_0}\right)^{r-1} \right).
$$
as $m \longrightarrow +\infty$. If $C_u \subseteq \overset{\circ}{\Delta_{\underline{a}}}$, the cardinal above is equivalent to $(m/m_0)^{r-1}$ as $m \longrightarrow + \infty$.
\item $\mathrm{card} \left( \{ u \in H_{m_0} \; | \; C_u \cap \partial (m_0 \cdot \Delta_{\underline{a}}) \neq \emptyset \} \right) = O(m_0^{r-2})$ as $m_0 \longrightarrow + \infty$.
\end{enumerate}
\end{lem}
\begin{proof}[Proof of the lemma] The first point is clear. The third point is easy to check since $\partial ( m_0 \cdot \Delta_{\underline{a}})$ is a union of $r - 2$ dimensional polyhedrons, and since all the $C_u$ are isometric, of diameter independent of $m_0$. 

Let us prove the second point. Let $C = \sum_{j=1}^{r-1} [0,1[ \cdot v_j$: this is a fundamental domain for the lattice $H$. If $C_u \subseteq \overset{\circ}{\Delta_{\underline{a}}}$, we have $C_u = C_u^\circ = u + C$, hence
\begin{align*}
(\mathbb R_+ \cdot C_u) \; \cap \; H_m & =   \mathbb R_+ \cdot \left( u + C \right) \cap H_m\\
& = \left[ \frac{m}{m_0} u +  \frac{m}{m_0} C \right] \cap H_m 
\end{align*} 
Thus 
\begin{align*}
\mathrm{card}\left((\mathbb R_+ \cdot C_u) \; \cap \; H_m\right) & = \mathrm{card} \left( \left[ \frac{m}{m_0} u +  \frac{m}{m_0} C \right] \cap H_m \right) \\
& = \mathrm{card}  \left( \frac{m}{m_0} C \cap H \right) + O(m^{r-2}),
\end{align*} 
where at the last line, we used the fact that $\partial H_m$ is a union of $r-2$ dimensional polyhedrons. Now, $C$ is a fundamental domain for $H$, so $\mathrm{card}(\frac{m}{m_0} C \cap H) \sim (\frac{m}{m_0})^{r-1}$ as $m \longrightarrow + \infty$. This ends the proof in the case where $C_u \subseteq m_0 \cdot \overset{\circ}{\Delta_{\underline{a}}}$. The proof of the general claim follows in the same lines, using $C_u \subseteq u + C$.
\end{proof}

\emph{Step 2. We apply Proposition \ref{propcone} to each cone $\mathbb R_+ \cdot C_u$.}
\medskip

Let $u \in H_{m_0}$. By construction, if $\underline{u_1}, ..., \underline{u_p}$ are the vertices of the polyhedron $C_u$, we have $\norm{\frac{u}{m_0} - \frac{1}{m_0} \, \underline{u_i}}_{\infty} \leq \frac{\max \norm{v_i}_{\infty}}{m_0} \leq \epsilon$. Thus, we are in the setting of Proposition \ref{propcone}: for any $m$, and any $(l_1, ..., l_r) \in (\mathbb R_+ \cdot C_u )\; \cap \; H_m  = (\sum_j \mathbb R_+ \cdot \underline{u_j}) \cap H_m$, we have
\begin{equation} \label{eqapplpropcone}
\chi^{[i]}(X, L_1^{\otimes l_1} \otimes ... \otimes L_r^{\otimes l_r}) \leq (-1)^i \left(\upsilon_{[\leq j]}(\frac{u}{m_0}) + O(\epsilon)\right) \frac{m^n}{n!} + O(m^{n-1}). 
\end{equation} 
where the constant in $O(\epsilon)$ do not depend on $m_0$ nor on $m$.
\medskip

\emph{Step 3. We sum over all cones $\mathbb R_+ \cdot C_u$.} Using Lemma \ref{lemauxiliary}, we can sort the $\underline{l} \in H_m$ among the cones $\mathbb R \cdot C_u$ to which they belong, and get:
\begin{align*}
\chi^{[i]} (X,  S^m (L_1^{(a_1)} & \oplus ... \oplus L_r^{(a_r)}))   = \sum_{\underline{l} \in H_m} \chi^{[i]} (X, L_1^{\otimes l_1} \otimes ... \otimes L_r^{\otimes l_r})  \\
 &  = \sum_{u \in H_{m_0}} \left(  \sum_{\underline{l} \in (\mathbb R_+ \cdot C_u) \; \cap \; H_m} \chi^{[i]} (X, L_1^{\otimes l_1} \otimes ... \otimes L_r^{\otimes l_r}) \right) \\
\end{align*}

Now, Step 2 permits to bound this from above by
	\begin{equation} \label{eqintermediate1}
(-1)^i \sum_{u \in H_{m_0}}  \left( \sum_{\underline{l} \in (\mathbb R_+ \cdot C_u) \cap H_m} 1 \right) \cdot\left[ (\upsilon_{[\leq j]} \left(\frac{u}{m_0} \right) + O(\epsilon) ) \frac{m^n}{n!} + O(m^{n-1}) \right]. 
\end{equation}

We can apply Lemma \ref{lemauxiliary} (2), (3) to get an upper bound by
	\begin{equation} \label{eqintermediate2}
		(-1)^i \sum_{u \in H_{m_0}}  \left(\frac{m}{m_0} \right)^{r-1} \cdot\left[ (\upsilon_{[\leq j]} \left(\frac{u}{m_0} \right) + O(\epsilon) ) \frac{m^n}{n!} + O(m^{n-1}) \right] + O \left( \frac{m}{m_0} \right)^{r-1} O(m_0^{r-2}) O(m^n) 
\end{equation}
To get this formula from \eqref{eqintermediate1}, we split the sum over $u \in H_{m_0}$ in two, distinguishing among the $u$ for which $C_u \cap \partial (m_0 \cdot \Delta_{\underline{a}}) \neq \emptyset$ (using Lemma \ref{lemauxiliary} (3) bounds this part of the sum by the second member of the formula above), and the other elements $u \in H_{m_0}$ (using Lemma \ref{lemauxiliary} (2), (3) bound this part of the sum by the full expression above).\medskip

Thus, we have proved that for any fixed $\epsilon > 0$, and any $m_0 > \frac{C}{\epsilon}$, we have
\begin{equation} \label{eqlimsup}
\limsup_{m \longrightarrow + \infty} \frac{\chi^{[i]} (X,  S^m (L_1^{(a_1)} \oplus ... \oplus L_r^{(a_r)}))}{m^{n+ r - 1}} \leq \frac{(-1)^i}{n!} \left( \frac{1}{m_0^{r-1}} \sum_{u \in H_{m_0}} \upsilon_{[\leq j]}\left( \frac{u}{m_0} \right) \right) + C_1 \epsilon + \frac{C_2}{m_0}.
\end{equation}
The constant $C_1$ does not depend on $m_0$. Indeed, in \eqref{eqintermediate2}, the constant in the $O(\epsilon)$ is independent of $m_0$, and we have $\frac{1}{m_0^{r-1}} \left( \sum_{u \in H_{m_0}} 1 \right) = \frac{1}{m_0^{r-1}} \mathrm{card}(H_{m_0}) \leq D$ for some constant $D$ depending only on $\underline{a}$. Also, the constant $C_2$ comes from the second member of \eqref{eqintermediate2} and does not depend on $\epsilon$.
\medskip

\emph{Step 4. We recognize a Riemann sum in the upper bound \eqref{eqlimsup}.}

As the element $u$ runs among $H_{m_0}$, the element $t = \frac{u}{m_0}$ runs among a lattice in $\Delta_{\underline{a}}$, with fundamental domain isometric to $\frac{1}{m_0} C$. The latter has euclidian volume $\frac{1}{m_0^{r-1}} \mathrm{vol}_{r-1}(C)$. Thus, as $m_0 \longrightarrow + \infty$, we have
\begin{align*}
\frac{1}{m_0^{r-1}} \sum_{u \in H_{m_0}} \upsilon_{[\leq j]} \left(\frac{u}{m_0}\right) & = \frac{ \mathrm{vol}_{r-1}(\frac{1}{m_0} C)}{\mathrm{vol}_{r-1}(C)} \sum_{u \in H_{m_0}} \upsilon_{[\leq j]} \left(\frac{u}{m_0}\right) \\
&  \underset{m_0 \longrightarrow + \infty}{\longrightarrow} \mathrm{vol}_{r-1}(C)^{-1} \int_{\Delta_{\underline{a}}} \upsilon_{[\leq j]} \, d \mathrm{vol}_{r-1}
\end{align*} 
Inserting this in \eqref{eqlimsup} and letting $\epsilon \longrightarrow 0$ and $m_0 \longrightarrow + \infty$, we get

$$
\limsup_{m \longrightarrow + \infty} \frac{\chi^{[i]} (X,  S^m (L_1^{(a_1)} \oplus ... \oplus L_r^{(a_r)}))}{m^{n+ r - 1}} \leq (-1)^i \mathrm{vol}_{r-1}(C)^{-1} \left( \int_{ \Delta_{\underline{a}}} \upsilon_{[\leq j]} \, d \mathrm{vol}_{r-1} \right) \frac{1}{n!}.
$$

To conclude, it suffices to use Lemma \ref{lemlattice2}, joint to the fact that $dP = \frac{1}{\mathrm{vol}_{r-1}(\Delta_{\underline{a}})} d \mathrm{vol}_{r-1}$. 
\end{proof}

\begin{rem} \label{lemthmineqroot}
	Using the modified version of Proposition \ref{propvariantmorse} mentioned in Remark \ref{remroot}, it is possible to adapt Theorem \ref{thmineqintegral} to the case where the $L_i$ are standard line bundles endowed with fractional trivializations $\frac{1}{d_i} \mathbf{e}_i$. The conclusion is unchanged, but we have to modify the definition of $\upsilon_{[\leq j]}$ by replacing the $m_i^{\mathfrak{s}}$ by $\frac{m_i^{\mathfrak{s}}}{d_i}$ in Definition \ref{defiupsilon}.
\end{rem}

\subsection{Twist by an auxiliary $\mathbb Q$-line bundle}

In the next section, we present the version of Theorem \ref{thmineqintegral} that we announced at the beginning of Section \ref{sectstatement}.

Let us recall the setting we introduced previously. We consider line bundles $L_1, ..., L_r$ on a complex variety $X$ of dimension $n$, as well as a stratification $\Sigma$ of $X$, with trivializations $\mathbf{e}_i$ of the $L_i$ over $\Sigma$. Assume now that we are given an auxiliary $\mathbb Q$-line bundle $N$ on $X$, so that $\Sigma$ is also adapted to $N$, and let $\frac{1}{d} \mathbf{g}$ be a fractional trivialization of $N$ over $\Sigma$. For each edge $\mathfrak{s}$ in the tree $\mathcal T$ associated to $\Sigma$, we let $m_i^{\mathfrak{s}}$ (resp. $p^{\mathfrak{s}}$) be the marking of $\mathfrak{s}$ associated to $\mathbf{e}_i$ (resp. $\mathbf{g}$). We now adapt Definition \ref{defiupsilon} to take into account our supplementary data.

\begin{defi} \label{defiupsilonN} Let $(t_1, ..., t_r) \in \mathbb R^r$. Mark each edge $\mathfrak{s}$ in $\mathcal T$ with the real number $t_1 m_1^{\mathfrak{s}} + ... + t_r m_r^{\mathfrak{s}} + \frac{1}{d} p^{\mathfrak{s}}$. For each complete path $\sigma$ in $\mathcal T$, denote by $C_{\sigma}$ the product of all markings along the edges of $\sigma$. Then, for all $i \in \llbracket 1, n \rrbracket$, we let
	$$
	\upsilon_{[\leq i]}^{N} (t_1, ..., t_r) = \sum_{\mathrm{index}(\sigma) \leq i} C_{\sigma},
	$$
	where the sum runs among all the complete path with a number of negative markings $\leq i$.
\end{defi}

We can now state the following corollary to Theorem \ref{thmineqintegral}.

\begin{corol} \label{corolineintegral} Let $\underline{a} = (a_1, ..., a_r) \in \mathbb N^r$, and let $P$ denote the uniform probability measure on $\Delta_{\underline{a}}$. Then, for all $i \in \llbracket 0, n \rrbracket$, and any $m$ divisible enough by $d$, we have the asymptotic upper bound
\begin{align*}
	\chi^{[i]} (X, N^{\otimes m} \otimes S^m (&L_1^{(a_1)} \oplus ...  \oplus L_r^{(a_r)}) )  \\
& \leq  \frac{\mathrm{gcd}(a_1, ..., a_r)}{a_1 ... a_r} \binom{n + r - 1}{r - 1} \left[ \int_{\Delta_{\underline{a}}} \upsilon_{[\leq i]}^N dP \right] \frac{m^{n+ r - 1}}{(n+ r - 1)!} + o(m^{n+r - 1}).
\end{align*}

	\begin{proof}

		\emph{Case 1. Assume first that $N$ is a standard line bundle, and that $d = 1$.} For each $1 \leq i \leq r$, we let $L_i' = L_i \otimes N^{\otimes a_i}$. By construction of the symmetric product of a weighted direct sum, we have, for any $m \in \mathbb N$:
		$$
		N^{\otimes m} \otimes S^m (L_1^{(a_1)} \oplus ... \oplus L_r^{(a_r)} ) = S^m (L_1' {}^{(a_1)} \oplus ... \oplus L_r'{}^{(a_r)} ).
		$$
		Hence, we can bound the $\chi^{[i]}$ of the term above by applying Theorem \ref{thmineqintegral}. To do this, we first need to produce a common trivialized stratification for all $L_i'$.
		\medskip
		
		Since $\Sigma$ is adapted to all $L_i$ and $N$, it is adapted to all $L_i'$. Furthermore, if $e_i$ (resp. $g$) is the trivialization of $L_i$ (resp. $N$) on a strata $U$ provided by $\mathbf{e}_i$ (resp. $\mathbf g$), we get a trivialization of $L_i'$ on $U$ by letting $e_i' = e_i \otimes g^{\otimes a_i}$. Let $\mathbf e'_i$ be the trivialization of $L_i'$ obtained by taking the $e'_i$ on all strata. Then, for any edge $\mathfrak{s}$ in $\mathcal T$, the marking of $\mathfrak{s}$ associated to $\mathbf{e}_i$ is equal to $m_i'{}^{\mathfrak{s}} := m_i^{\mathfrak{s}} + a_i \, p^{\mathfrak{s}}$. 

		\medskip
		
		We can now use Definition \ref{defiupsilon} to define the function $\upsilon_{[\leq i]}$ associated to the data $(\Sigma, (\mathbf e'_i)_{1 \leq i \leq r})$. Let $\underline{t} = (t_1, ..., t_r) \in  \Delta_{\underline{a}}$, and let $\mathfrak{s}$ be an edge of $\mathcal T$. Then Definition \ref{defiupsilon} prescribes to mark $\mathfrak{s}$ with the weight $t_1 m_1'{}^{\mathfrak{s}} + ... + t_r m_r'{}^{\mathfrak{s}}$.
		Since $\underline{t} \in \Delta_{\underline{a}}$, we have $a_1 t_1 + ... + a_r t_r = 1$, and the previous weight is equal to $t_1 m_1 {}^{\mathfrak{s}} + ... + t_r m_r{}^{\mathfrak{s}} + p^{\mathfrak{s}}$. This shows that the function $\upsilon_{[\leq j]}$ coincides on $\Delta_{\underline{a}}$ with the function $\upsilon^N_{[\leq j]}$ introduced in Definition \ref{defiupsilonN}. This reduces the required inequality to an application of Theorem \ref{thmineqintegral} to the line bundles $L_i'$.

		\medskip

		\emph{General case. }In the setting where $L$ is a $\mathbb Q$-line bundle, with $d > 1$, we can adapt the proof using the following instructions.
		\begin{enumerate}
			\item Perform a Bloch-Gieseker covering to reduce to the case where $L$ is a standard line bundle;
			\item  Construct the $L_i'$ as above;
			\item Instead of applying Theorem \ref{thmineqintegral} to the $L_i'$, apply its modified version mentioned in Remark \ref{lemthmineqroot}. According to this definition, we have to mark the $\mathfrak{s}$ with $m_i'{}^\mathfrak{s} =  m_i^{\mathfrak{s}} + a_i \, \frac{p^{\mathfrak{s}}}{d}$. The same computation as above yields the result.
		\end{enumerate}

	\end{proof}

\end{corol}

\section{Proof of the main theorem}

\subsection{Statement of the result} \label{sectstatement}

We will now apply Corollary \ref{corolineintegral} to the situation of a direct sum $E^{(1)} \oplus ... \oplus E^{(k)}$, where $E$ is a direct sum of line bundles. The main result of this section is an algebraic version of the main theorem \cite[Theorem 2.37]{dem11}. Before stating it, let us introduce a simplifying notation. 

\begin{bfseries}Notation.\end{bfseries} If $E$ is a vector bundle, and $k$ an integer, we will denote by $\mathbb E_k$ the weighted direct sum
	$$
	\mathbb E_k = E^{(1)} \oplus ... \oplus E^{(k)}.
	$$

\begin{thm} \label{thmcomparison} Let $X$ be a complex projective manifold of dimension $n$, and let $E = L_1 \oplus ... \oplus L_r$ be a direct sum of line bundles over $X$. For $k \in \mathbb N^\ast$. Let $N$ be an auxiliary line bundle on $X$. For each $k$, we introduce the $\mathbb Q$-line bundle
	$$
	N_k = \mathcal O_X \left( \frac{1}{kr} ( 1 + \frac{1}{2} + ... + \frac{1}{k} ) F \right).
	$$

Assume that $\Sigma$ is a stratification of $X$, adapted to $\mathrm{det}\, E \otimes N = L_1 \otimes ... \otimes L_r \otimes N$, and let $\mathbf{e}$ be a trivialization of $\det\, E \otimes N$ over $\Sigma$. Let $\underline{\Sigma} = (\Sigma, \mathbf{e})$.
\medskip

Then, for all $j \in \llbracket 0, n \rrbracket$, and all $m \gg k \gg 1$, with $m$ divisible enough, we have
\begin{align*}
	\chi^{[j]} (X, S^m \mathbb E_k \, \otimes \, N_k^{\otimes m })  \leq (-1)^j \frac{(\log k)^n}{n! (k!)^r} & \left( c_1 (\det E \otimes N, \underline{\Sigma})^n_{[\leq j]} + O(\frac{1}{\log k}) \right) \frac{m^{n+kr - 1}}{(n + kr - 1)!}  \\
	& + o(m^{n+kr-1})
\end{align*}
\end{thm}

Let us place ourselves in the hypotheses of the theorem, and define a few objects that will be useful in the proof. 

\begin{defi} \label{defitrivial} By Propositions \ref{proprefine} and \ref{proprefine2}, we do not lose generality in assuming that $\Sigma$ is also adapted to $L_1, ..., L_r, N$. Under this hypothesis, we introduce trivializations $\mathbf{e}_1, ..., \mathbf{e}_r, \mathbf g$ of $L_1, ..., L_r, N$ on $\Sigma$, such that the following holds. For any irreducible component $V$ appearing in the stratification $\Sigma$, if $U \subseteq V$ is the complement of the natural strata on $V$, and if $e_i$ (resp. $g, e$) is the trivialization of $L_i$ (resp. $N$, resp. $\det E \otimes N$) on $U$ given by $\mathbf{e}_i$ (resp. $\mathbf{g}$, resp. $\mathbf{e}$), we have
$$
e_1 \otimes ... \otimes e_r \otimes g = e.
$$  
\end{defi}

\begin{rem} It is always possible to find $\mathbf{e}_1, ..., \mathbf{e}_r, \mathbf{g}$ as in Definition \ref{defitrivial}, by first fixing $e_1, ..., e_{r}, e$ on $U$, and then letting $g = e \cdot (e_1)^{-1} \cdot ... \cdot (e_{r})^{-1}$.
\end{rem}

Let $\underline{k} = (1, ..., 1, 2, ..., 2, ...., k, ..., k)$, where each number is repeated $r$ times. Applying Corollary \ref{corolineintegral} to the weighted direct sum $\mathbb E_k = L_1^{(1)} \oplus ... \oplus L_r^{(1)} ... \oplus L_1^{(k)} \oplus ... \oplus L_r^{(k)}$ and to the $\mathbb Q$-line bundle $N_k$, we get
\begin{align} \nonumber 
	\chi^{[i]} (X, N_k^{\otimes m } \; \otimes \; & S^m \mathbb E_k) \\ \label{equpperboundint}
	& \leq \frac{1}{(k!)^r} \binom{n+ kr - 1}{kr - 1} \left[ \int_{ \Delta_{\underline{k}}} \upsilon^{N_k}_{[\leq j]} \, d P \right] \frac{m^{n + kr - 1}}{(n+ kr - 1)!} + o(m^{n+ kr - 1})
\end{align}
where $\upsilon^{N_k}_{[\leq j]} : \mathbb R_+^{kr} \longrightarrow \mathbb R$ is the function provided by Definition \ref{defiupsilonN}, and $dP$ is the uniform probability measure on $\Delta_{\underline{k}}$.
\medskip

The theorem will come directly from the following asymptotic estimate of the integral term.

\begin{prop} \label{propasymptint} We have, as $ k \longrightarrow + \infty$,
\begin{equation} \label{eqasymptint}
	\int_{ \Delta_{\underline{k}}} \upsilon^{N_k}_{[\leq j]} \, d P = \frac{(\log k)^n}{(kr)^n} \left[ c_1(\det E \otimes N, \underline{\Sigma})^n_{[\leq j]} + O(\frac{1}{\log k}) \right].
\end{equation}
\end{prop}

This proposition implies Theorem \ref{thmcomparison} right away: it suffices to insert \eqref{eqasymptint} in \eqref{equpperboundint}, and to remark that $\frac{1}{(kr)^n} \binom{n + kr - 1}{kr - 1} = \frac{1}{n!} (1 + O(\frac{1}{k}))$, with $n, r$ fixed, and $k \longrightarrow + \infty$.
\medskip

\subsection{Proof of Proposition \ref{propasymptint}}

We propose to further elaborate on Demailly's Monte-Carlo approach, and to interpret the integral in \eqref{eqasymptint} as the mean value of the random variable $\upsilon^{N_k}_{[\leq j]}$ depending of a uniform sorting in $\Delta_{\underline{k}}$. The reader should compare \eqref{equpperboundint} with \cite[(2.17)]{dem11}: even though the computations are closely related, our asymptotic estimate is slightly different to the one of Demailly, as our random variables will depend on random sorting inside $\Delta_{\underline{k}}$, and not on a product $\Delta^{k-1} \times (S^{2r - 1})^k$. We refer to Section \ref{sectprob} for some useful computations related to uniform random variables on simplexes.
\medskip

Let $\mathcal T$ the tree associated to $\Sigma$, and let $\sigma$ be a complete path in $\mathcal T$. For all $i \in \llbracket 1, n \rrbracket$, we denote by $V_i^\sigma$ the irreducible $i$-dimensional variety that appears along the labels of $\sigma$ (see Remark \ref{remtree}). We also denote by $f_i : V_i^\sigma \longrightarrow V_{i+1}^\sigma$ the natural map provided by the stratification. Now, for all $i \in \llbracket 1, n \rrbracket$, and all $j \in \llbracket 1, r \rrbracket$, denote by $d_j^i(\sigma)$ the multiplicity along $f_{i-1}(V_{i-1})$ of the trivialization of $L_j$  provided by $\mathbf{e}_j$. Also, let $d^i(\sigma)$ (resp. $p^i(\sigma)$) denote the multiplicity of the trivialization of $\det E = L_1 \otimes ... \otimes L_r \otimes N$ (resp. $N$) provided by $\mathbf{e}$ (resp. $\mathbf{g}$) along $f_{i-1}(V_{i-1}^\sigma)$.

By our choice of $\mathbf{e}_1, ..., \mathbf{e}_r$ and $\mathbf{e}$ in Definition \ref{defitrivial}, the following property is straightforward.

\begin{lem} \label{lemsum}
	For all complete path $\sigma$ in $\mathcal T$, and for all $i \in \llbracket 1, n \rrbracket$, we have $d^i(\sigma) = d_1^i(\sigma) + ... + d_r^{i}(\sigma) + p^i(\sigma)$.
\end{lem}

In this setting, Definition \ref{defiupsilonN} prescribes to compute $\upsilon^{N_k}_{[\leq j]}$ as follows. Let 
$$
t = (t_{j, l})_{1 \leq j \leq k, 1 \leq l \leq r} \in \Delta_{\underline{k}}.
$$
For all $i \in \llbracket 1, n \rrbracket$ and all complete path $\sigma$, mark the edge from $V_{i-1}^\sigma$ to $V_i^\sigma$ with the real number 
$$
A_{i}^\sigma (t) = \sum_{ 1 \leq j \leq k, 1 \leq l \leq r} t_{k, l} d_l^{i}(\sigma) + p'{}^i(\sigma). 
$$
where $p'{}^i(\sigma) = \left[ \frac{1}{kr} \left( 1 + \frac{1}{2} + ... \frac{1}{k} \right) \right]^{-1} p^i(\sigma)$.

Then, we have
\begin{equation} \label{equpsilon}
	\upsilon^{N_k}_{[\leq j]} (t) = \sum_{\sigma \; \text{complete path}} \left[ \mathbbm 1_{\{\mathrm{index}(\sigma) \leq j \}} \prod_{1 \leq i \leq n} A_{i}^\sigma(t) \right],
\end{equation}
where $\mathbbm 1_{\{\mathrm{index}(\sigma) \leq j \}} = 1$ if there are less that $j$ negative values among the $A_i^\sigma(t)$ ($1 \leq i \leq n$), and $0$ otherwise. Note that this index depends on $t$: we will not write explicitly this dependence to lighten a bit the notations.
\medskip

We will now interpret each $A_{i}^\sigma(t)$, as well as $\upsilon^{N_k}_{[\leq j]}(t)$, as a random variable, using the probability measure $dP$ to draw a random element $t \in \Delta_{\underline{k}}$. To simplify the presentation, let us fix a complete path $\sigma$, and remove it for the time being from our notations. In the next lemma, we give estimates on the expectancy value and the variance of the $A_i$.

\begin{lem} \label{lemvariance} Let $i \in \llbracket 1, n \rrbracket$. Then,
\begin{enumerate}
	\item the expectancy value satisfies $A_i$ is $\mathbf{E}(A_i) \sim \frac{\log k}{kr} d^i$ as $k \longrightarrow + \infty$.
\item There is a constant $C_i$ depending only on the $d^i_l$ $(1 \leq l \leq r)$, such that $$\mathbf{Var}(A_i) \leq \frac{C_i}{k^2} V.$$
\end{enumerate}

\end{lem}
\begin{proof}
	We are in the situation of Section \ref{sectprob} : $t = (t_{j, l})_{1 \leq j \leq r\, 1 \leq l \leq k}$ is drawn with uniform law in the simplex $\Delta_{\underline{k}}$, and $A_i(t)$ is an affine function of $t$ of the form $A_i(t) = \sum_{1 \leq j \leq k} \sum_{1 \leq l \leq r} t_{j, l} d^i_l + p'{}^i$.

	\emph{(1)} Since $A_i$ is an affine function, it is easy to see that its mean value on $\Delta_{\underline{k}}$ is equal to the average value of the images of the vertices of $\Delta_{\underline{k}}$ by $A_i$. These vertices are the $\mathrm{v}_{j,k} = (\frac{1}{l} \delta_{j, j'} \delta_{l, l'})_{1 \leq j' \leq k, 1 \leq l' \leq r}$ for $1 \leq j \leq k$ and $1 \leq l \leq r$. The affine function $A_i$ takes the value $\frac{1}{j} d^i_l + p'{}^i$ on $\mathrm{v}_{j,k}$. Thus
	\begin{align*}
		\mathbf{E}[A_i] & = \frac{1}{k r}  \sum_{1 \leq l \leq r} \sum_{1 \leq j \leq k} (\frac{1}{j} d^i_l + p'{}^i) \\
				& = \frac{1}{kr} \sum_{1 \leq j \leq k} \frac{1}{j} \left( \sum_{1 \leq l \leq r} d^i_l + p^i \right) 
	\end{align*}
	This gives the first point, since $\sum_{1 \leq l \leq r} d^i_l + p^i = d^i$ by Lemma \ref{lemsum}.

\emph{(2)} This follows right away from Lemma \ref{lemboundvariance}: if $M$ is the mean value of the random variable $\left( \sum_{1 \leq l \leq r} T_l d_l \right)^2$, with $T$ uniformly distributed in $\Delta^{r-1}$, the lemma provides the result with $C_i = \frac{\pi^2}{3} M$. 
\end{proof}

We now state the fundamental lemma that will allow us to end the proof of Proposition \ref{propasymptint}: it can be seen as a version of \cite[Lemma 2.25]{dem11}, adapted to our combinatorial context. Recall that we are working on a fixed path $\sigma$. We let $j_\sigma$ be the index of this complete path for the trivialization $\mathbf e$, i.e. the number of negative labels among the $d^i = d^i(\sigma)$.

\begin{lem} \label{lemupperbounddiff}
Let $j \in \llbracket 0, n \rrbracket$. Then, we have
\begin{align} \nonumber
| \; \mathbf{E}( \mathbbm 1_{\{\mathrm{index(\sigma)} = j\} } \prod_{1 \leq i \leq n} A_i ) -  \delta_{j, j_\sigma} & \prod_{1 \leq i \leq n}   \mathbf E(A_i) \; | \\ \label{equpperbounddiff}
& \leq \left[ \sum_{1 \leq p \leq n} \; \; \left( \prod_{1 \leq q \leq p - 1} \mathbf{E} (A_q^2) \right)  \; \mathbf{Var}(A_p) \left( \prod_{p+1 \leq s \leq n} \mathbf{E}(A_s)^2 \right) \right]^{1/2}
\end{align}
\end{lem}

\begin{proof} The proof is essentially the same as \cite[Lemma 2.25]{dem11}, and is based on the following observation: let $(a_1, ..., a_t), (b_1, ..., b_t) \in \mathbb R^t$ be such that there are exactly $\alpha$ negative numbers among the $a_i$, and $\beta$ negative numbers among the $b_i$. Then we have, for any $j$:
$$
|\mathbbm{1}_{\{j=\alpha\}} \prod_{i} a_i - \mathbbm{1}_{\{j=\beta\}} \prod_{i} b_i| \leq \sum_{1 \leq p \leq t} \left( \prod_{1 \leq q \leq p-1} |a_q| \right) |a_p - b_p| \left( \prod_{p + 1 \leq s \leq t} |b_s| \right).
$$
This is easy to show by distinguishing among the possible values of $j$.

This observation gives:
\begin{align*}
| \mathbbm{1}_{\{\mathrm{index}(\sigma) = j \}} \prod_i A_i -  \mathbbm 1_{\{ j = j_\sigma \} } & \prod_{i} \mathbf E(A_i) | \leq \\
& \leq \sum_{1 \leq p \leq n} \; \; \left( \prod_{1 \leq q \leq p - 1} | A_q | \right)  \; | A_p - \mathbf{E}(A_p) | \; \left( \prod_{p+1 \leq s \leq n} \mathbf{E}(A_s) \right)
\end{align*}

Taking the expectancy value, we obtain
\begin{align*}
| \; \mathbf{E}( \mathbbm 1_{\{\mathrm{index(\sigma)} = j\} } \prod_i A_i) - \mathbbm 1_{\{ j = j_\sigma \} } & \prod_{i} \mathbf E(A_i) \; |^2\\
& \leq \mathbf{E} \left(| \mathbbm{1}_{\{\mathrm{index}(\sigma) = j \}} \prod_i A_i -  \mathbbm 1_{\{ j = j_\sigma \} } \prod_{i} \mathbf E(A_i) | \right)^2 \\
& \leq \mathbf{E} \left( \sum_{1 \leq p \leq n} \; \; \left( \prod_{1 \leq q \leq p - 1} | A_q | \right)  \; | A_p - \mathbf{E}(A_p) | \; \left( \prod_{p+1 \leq s \leq n} \mathbf{E}(A_s) \right) \right)^2 \\
& \leq \sum_{1 \leq p \leq n} \; \; \left( \prod_{1 \leq q \leq p - 1} \mathbf{E} (A_p^2) \right)  \; \mathbf{E} (|A_p - \mathbf{E}(A_p)|^2) \; \left( \prod_{p+1 \leq s \leq n} \mathbf{E}(A_s)^2 \right),
\end{align*}
where, at the last line, we used Cauchy-Schwarz inequality $\mathbf{E}(XY)^2 \leq \mathbf{E}(X^2) \mathbf{E}(Y^2)$. Since $\mathbf{Var}(A_l) = \mathbf{E} (|A_l - \mathbf{E}(A_l)|^2)$, we get the result.
\end{proof}

We are now ready to end the proof of Proposition \ref{propasymptint}, by summing the previous estimates over all paths $\sigma$.

\begin{proof}[Proof of Proposition \ref{propasymptint}]

Using \eqref{equpsilon}, we can write
\begin{align*}
	\int_{\Delta_{\underline{k}}} \upsilon^{N_k}_{[\leq j]} d P & =  \mathbf{E} [ \upsilon^{N_k}_{[\leq j]} ]\\
	&  = \sum_{\sigma}  \mathbf{E} [ \mathbbm{1}_{\{\mathrm{index}(\sigma) \leq j\}} \prod_{1  \leq i \leq n} A_i^\sigma ] \\
	& \leq  \sum_{\sigma} \left[\mathbbm 1_{\{j_\sigma \leq j\}} \prod_{1 \leq i \leq n} \mathbf{E}[A_i^\sigma] \right]  + \sum_\sigma \left[ \sum_{1 \leq l \leq j} \left| \mathbf{E}( \mathbbm 1_{\{\mathrm{index(\sigma)} = l\} } \prod_{1 \leq i \leq n} A_i^\sigma ) -  \delta_{l, j_\sigma} \prod_{1 \leq i \leq n}   \mathbf E(A_i^\sigma) \; \right| \right] 
\end{align*}

We can apply Lemma \ref{lemupperbounddiff} to bound from above the second term of the right hand side. By Lemma \ref{lemvariance}, (1), (2), the right hand side of \eqref{equpperbounddiff} is bounded from above by a term of the form $C \frac{(\log k)^{n-1}}{k^n}$, where the constant $C$ does not depend on $k$. Thus, we get:
$$\int_{\Delta_{\underline{k}}} \upsilon^N_{[\leq j]} d P =  \sum_\sigma \mathbbm 1_{\{j_\sigma \leq j\}} \prod_{1 \leq i \leq n} \mathbf{E}[A_i^\sigma] + O\left(\frac{(\log k)^{n-1}}{k^n} \right) 
$$
Using again Lemma \ref{lemvariance}, (1), we then obtain
$$
\int_{\Delta_{\underline{k}}} \upsilon^N_{[\leq j]} d P = \frac{(\log k)^n}{(kr)^n} \left[ \sum_\sigma \left( \mathbbm 1_{\{j_\sigma \leq j\}} \prod_{1 \leq i \leq n} d^i(\sigma) \right) \right] + O(\frac{(\log k)^{n-1}}{k^n}).
$$

	Now, the term between brackets is equal to $c_1(\det E \otimes N, \underline{\Sigma})_{[\leq j]}^n$, since the $d^i(\sigma)$ are the multiplicities of the trivialization $\mathbf e$ along the strata of $\Sigma$. This concludes the proof of Proposition \ref{propasymptint}, and of Theorem \ref{thmcomparison}.
\end{proof}

\subsection{End of the proof of the main theorem}

Theorem \ref{thmprinc} follows directly from Proposition \ref{proptwisteestimate} and the following result. Again, if $E$ is a vector bundle, we write $\mathbb E_k = E^{(1)} \oplus ... \oplus E^{(k)}$.

\begin{prop} \label{propfinal}
Let $X$ be a smooth complex projective manifold of dimension $n$. Let $E \longrightarrow X$ be a vector bundle of rank $r$, such that $\det E$ is big. For any $\epsilon > 0$, there exists 
	\begin{enumerate} \itemsep=0em
		\item a generically finite projective morphism $p : X' \longrightarrow X$;
		\item a decomposition $p^{\ast} \det E = A + G$ into ample and effective divisors;
		\item a trivialized stratification $\underline{\Sigma}$ on $X'$, adapted to $A$, such that $$\deg c_1(A, \underline{\Sigma})^n_{[\leq 1]} > (\deg p) (\mathrm{vol}(K_X) - \epsilon) > 0;$$
		\item a sequence of effective $\mathbb Q$-divisors $(F_k)_{k \geq 1}$ on $X'$; 
	\end{enumerate}
			
			such that the following holds.  
\medskip

For $m \gg k \gg 0$, and $m$ divisible enough, we have,
\begin{align} \nonumber
	\chi^{[1]} \left( X', \right. & \left.  S^m \left( \mathbb E_k \right) \otimes \mathcal O(-m \, F_k) \right) \\ \label{eqineqpropfinal}
	& \leq \; \frac{m^{n + kr - 1}}{(n + kr - 1)!} \frac{(\log k)^n}{(k!)^n} \left( \deg c_1(A, \underline{\Sigma})^n_{[\leq 1]} - O(\frac{1}{\log k}) \right) + o (m^{n + kr - 1})
\end{align}
\end{prop}

Before proving the proposition above, let us explain how it permits to prove Theorem \ref{thmfulldetail}, and thus Theorem \ref{thmprinc}.

\begin{proof}
Let $X$ be a projective manifold of dimension $n$, such that $K_X = \det \Omega_X$ is big.

	Let $\epsilon > 0$. We now apply Proposition \ref{propfinal} with $E = \Omega_X$, to obtain $p : X' \longrightarrow X$ and $(F_k)_{k \geq 1}$ such that \eqref{eqineqpropfinal} holds. Then, Proposition \ref{proptwisteestimate} implies in turn that
\begin{align*}
h^0 (X', p^\ast & E_{k,m}^{GG}) \geq 
	\; \frac{m^{n + kr - 1}}{(n + kr - 1)!} \frac{(\log k)^n}{(k!)^n} \left( \deg c_1(A, \underline{\Sigma})_{[\leq 1]}^n - O((\log k)^{-1}) \right) \\
 & \hspace{1cm} + o (m^{n + kr - 1}).
\end{align*}    

	This implies that $p^\ast E_{k, \bullet}^{GG} \Omega_{X}$ is big. Since $p$ is generically finite of degree $\mathrm{deg}(p)$, we have, by Lemma \ref{lemmodification},
$$
	\mathrm{vol}(X', p^\ast E_{k, \bullet }^{GG} \Omega_X) = \deg (p) \cdot  \mathrm{vol}(X, E_{k, \bullet }^{GG} \Omega_X)
$$
	so $E_{k, \bullet}^{GG} \Omega_X$ is big, and this ends the proof.
\end{proof}

Since Theorem \ref{thmfulldetail} holds for all $\epsilon > 0$, we get the following corollary (implied by \cite[Corollary 2.38]{dem11}).  

\begin{corol} \label{thmvol} Let $X$ be a complex projective manifold, with $K_X$ big. For $k \gg 1$, we have:
$$
\mathrm{vol}(E_{k, \bullet}^{GG} \Omega_X) \geq \frac{(\log k)^n}{(k!)^n} \left( \mathrm{vol}(K_X) - O\left( \frac{1}{\log k} \right) \right).
$$
\end{corol}

We now finish with the proof of Proposition \ref{propfinal}, which is based on Theorem \ref{thmcomparison}. We are essentially looking for a way to construct a natural stratification adapted to a line bundle of the form $A = \det E \otimes N$. If the latter were very ample, this would be easily done by taking successive generic hyperplane sections. In the ample case, we can use the same idea, but we need to pass to a ramified cover ; this technique provides the following result. 

\begin{lem} \label{lemstrat}
	Let $X$ be a complex projective manifold of dimension $n$, and let $A$ be an ample line bundle on $X$.  Then there exists a finite dominant morphism $X' \overset{p}{\longrightarrow} X$ and a trivialized stratification $\underline{\Sigma} = (\Sigma, \mathbf{e})$ on $X'$, adapted to $p^\ast A$, such that
	$$
	\deg c_1(p^\ast A, \underline{\Sigma})_{[\leq 1]}^n = (\deg p) \, (A^n).
	$$
\end{lem}
\begin{proof}
	Let $m \in \mathbb N$ be such that $B = A^{\otimes m}$ is very ample. Using Bloch-Gieseker lemma \cite{BG71, KM98}  (see also \cite[Theorem 4.1.10]{lazpos1}), we find $p_1 : X_1 \longrightarrow X$, finite, dominant, such that $p_1^\ast B = C^{\otimes m}$, with $C$ very ample (the fact that $C$ can be chosen very ample follows directly from the proof presented e.g. in \cite{lazpos1}). \medskip

	Since $C \longrightarrow X$ is very ample, it induces an embedding $X_1 \hookrightarrow \mathbb P^N$. Intersecting $X_1$ with a generic flag $\mathbb P^{N - n} \subseteq \mathbb P^{N - n + 1} \subseteq ... \subseteq \mathbb P^N$, and using the standard trivialization of $\mathcal O(1)$ on the successive $\mathbb P^{N-i} \setminus \mathbb P^{N-i-1}$, we get a trivialized stratification $\underline{\Sigma_1}$ on $X_1$, adapted to $C$. This stratification satisfies
$$
c_1 (C, \underline{\Sigma_1})^n_{[\leq 1]} = c_1(C, \underline{\Sigma}_1)^n = (C^n). 
$$
	where the first equality holds since all paths appearing in the graph associated to $\underline{\Sigma_1}$ have only positive markings (these markings are in fact all equal to $1$).
\medskip

	The line bundle $L = p_1^\ast A \otimes C^{\otimes (-1)}$ is such that $L^{\otimes m} = \mathcal O_{X_1}$, so there exists a finite étale covering $p_2 : X' \longrightarrow X_1$ of degree $m$ such that $p_2^\ast L \cong \mathcal O_{X'}$. Thus, we have 
\begin{equation} \label{eqroot}
p_2^\ast p_1^\ast A \cong p_2^\ast C.
\end{equation}
\medskip

	We let $p = p_1 \circ p_2$. The morphism $X' \overset{p}{\longrightarrow} X$ will be the required finite morphism ; we just have to exhibit the trivialized stratification $\underline{\Sigma}$ on $X'$. Taking the fiber products between $p_2 : X_2 \longrightarrow X_1$ and the irreducible components of the strata of $\Sigma_1$, we obtain a trivialized stratification $\underline{\Sigma}$ on $X'$, adapted to $p_2^\ast C = p^\ast A$. This implies then $c_1(p^\ast A,  \underline{\Sigma})^n_{[\leq 1]} = c_1(p_2^\ast C, \underline{\Sigma})^n = (p_2^\ast C^n)$.
	
	A repeated application of the projection formula finally yields
	\begin{align*}
		c_1 (p^\ast A, \underline{\Sigma})^n_{[\leq 1]}	& = (\deg p_2) (C^n) \\
		& = (\deg p_2) \frac{(\deg p_1)}{m^n} (B^n) \\
		& = (\deg p) (A^n),
	\end{align*}
	since $(\deg p_1) (\deg p_2) = \deg p$, and $(B^n) = m^n (A^n)$.
\end{proof}

\begin{proof}[Proof of Proposition \ref{propfinal}]
Since $L = \det E$ is big, we can use Fujita's approximate Zariski decomposition theorem \cite{fuj94, DEL00}, to obtain a modification $p_1 : X_1 \longrightarrow X$ and an integer $m$ such that
$$
m (p_1^\ast L) = A_1 + G_1, 
$$
where $A_1$ is ample with $(A_1^n) \geq m^n \left( \mathrm{vol} (\det E) - \epsilon \right)$, and where $G_1$ is an effective divisor.

Now, we can use Kawamata's covering lemma \cite{kawa82} (see also \cite[Proposition 4.1.12]{lazpos1}) to find a finite dominant morphism $p_2 : X_2 \longrightarrow X_1$ such that $p_2^\ast G_1 = mG_2$, where $G_2$ is effective. Then, if $q = p_2 \circ p_1$, we have
$$
m (q^\ast L - G_2) = p_2^\ast A_1
$$
	The divisor $p_2^\ast A_1$ is ample as the pullback of an ample divisor by a finite morphism, and $(p_2^\ast A^n) = (\deg p_2) (A^n)$. Thus, $q^\ast L - G_2$ is itself ample, and by Lemma \ref{lemstrat}, there exists a finite dominant morphism $p_3 : X' \longrightarrow X_2$ and a trivialized stratification $\Sigma$, adapted to the ample divisor $A = p_3^\ast (q^\ast L - G_2)$, such that $\deg c_1(A, \underline{\Sigma})_{[\leq 1]}^n = (A^n)$.
\medskip

Let $p = q \circ p_3$. The intersection number above can be computed by repeated applications of the projection formula:
	\begin{align*}
		(A^n) & = \frac{1}{m^n} (p_3^\ast p_2^\ast A_1^n) \\
					      & = \frac{1}{m^n} (\deg p_3)(\deg p_2) (A_1^n) \\ 
					      & \geq (\deg p) \left( \mathrm{vol} (\det E) - \epsilon \right),
	\end{align*}
	since $p = p_1 \circ p_2 \circ p_3$, with $\deg p_1 = 1$. We have the requested inequality: 
	\begin{equation} \label{ineqc1}
		\deg c_1(A, \underline{\Sigma})_{[\leq 1]}^n \geq (\deg p) \left( \mathrm{vol} (\det E) - \epsilon) \right).
	\end{equation}

	To conclude, we let $F_k = \frac{1}{kr} \left( 1 + \frac{1}{2} + ... + \frac{1}{k} \right) p_3^\ast G_2$ for all $k \geq 1$. We now apply Theorem \ref{thmcomparison} with $X$ replaced by $X'$, $E$ replaced by $p^\ast E$ and letting $N = \mathcal O(- p_3^\ast G_2)$. We have then $N_k = \mathcal O(-F_k)$ and $\det (p^\ast E) \otimes N = \mathcal O(A)$, so Theorem \ref{thmcomparison} gives the result immediately.
\end{proof}

\begin{rem}
	As it was the case in \cite{dem11}, it is actually not necessary to use Fujita's approximation's theorem to get Theorem \ref{thmprinc}, if we are not interested in the more precise volume estimate of Theorem \ref{thmvol}.

	In the proof of Proposition \ref{propfinal}, it suffices to take $X_1 = X$ and $p_1 = \mathrm{Id}_{X}$, and to remark that since $L$ is big, then $m L = A + G_1$ for some $m \gg 1$, $A$ ample, and $G_1$ effective. Then $(A^n) > 0$, and this is enough to find $\underline{\Sigma}$ so that $\deg c_1 (A, \underline{\Sigma})^n_{[\leq 1]} > 0$. The final estimate of Proposition \ref{propfinal} is still valid, and this is enough to prove Theorem \ref{thmprinc}.
\end{rem}

\section{Annex. Some computations on simplexes}

For the convenience of the reader, we gather here a few classical or technical results and computations which were used in the rest of the text.

\subsection{Lattices and volumes of fundamental domains}

Let $a_1, ..., a_r \in \mathbb N$. Let $H = \{ (t_1, ..., t_r) \in \mathbb Z^r \; | \; \sum_i a_i t_i = 0\}$. Then $H \subseteq \mathbb Z^r$ is a \emph{primitive} sublattice, meaning that $\quotientd{\mathbb Z^r}{H}$ is torsion-free. Hence, by the adapted basis theorem, there exists a basis $(f_1, ..., f_r)$ of $\mathbb Z^r$ such that $(f_1, ..., f_{r-1})$ is a basis of $H$. Let $C_H = \sum_{1 \leq i \leq r-1} [0, 1] \cdot f_i$ denote the associated fundamental domain of $H$.
\medskip

For all $n$, we let $\mathrm{vol}_{n}$ denote the $n$-dimensional euclidian volume measure.

\begin{lem} \label{lemlattice1} The fundamental domain of $H$ has volume $\mathrm{vol}_{r-1} \left( C_H \right) = \frac{\sqrt{ \sum_{1 \leq i \leq r } a_i^2}}{\gcd(a_1, ..., a_r)}$.
\end{lem}
\begin{proof}
The lattice $H$ and its fundamental domain do not change if we replace $a_i$ by $\frac{a_i}{\gcd(a_1, ..., a_r)}$, hence we can suppose that $\gcd(a_1, ..., a_r) = 1$. In this case, there exist $u_1, ..., u_r \in \mathbb Z$ such that $\sum_i a_i u_i = 1$, and we can assume that $f_r = (u_1, ..., u_r)$.

Since $(f_1, ..., f_{r})$ is a basis of $\mathbb Z^r$, we have $\mathrm{vol}_r(\sum_{1 \leq i \leq r} [0, 1] \cdot f_i) = 1$. Moreover, 
\begin{align*}
\mathrm{vol}_r(\sum_{1 \leq i \leq r} [0, 1] \cdot f_i) & = \mathrm{vol}_{r-1}(\sum_{1 \leq i \leq r - 1} [0, 1] \cdot f_i) \, \cdot \, \norm{\pi_{H^\perp} (f_r) }_{\mathrm{eucl}} \\
	& = \mathrm{vol}_{r-1}(C_H) \, \cdot\, \norm{\pi_{H^\perp} (f_r) }_{\mathrm{eucl}}
\end{align*}
where $\pi_{H^\perp} (f_n)$ is the orthogonal projection of $f_r$ on $H^\perp$, and $\norm{ \cdot }_{\mathrm{eucl}}$ is the euclidian norm. Since $H^\perp = \mathbb R \cdot (a_1, ..., a_r)$ by definition of $H$, a direct computation gives $\norm{\pi_{H^\perp} (f_r) }_{\mathrm{eucl}} = \frac{1}{\norm{(a_1, ..., a_p)}_{\mathrm{eucl}}}$, hence the result.
\end{proof}

For the next lemma, we resume the notations introduced in Definition \ref{defisimplex}.

\begin{lem} \label{lemlattice2} Let $\underline{a} = (a_1, ..., a_r) \in \mathbb N^r$, and let $\Delta_{\underline{a}} = \{ (t_i) \in \mathbb R_+^r \; | \; \sum_i a_i t_i = 1 \}$. Then the volume of $\Delta_{\underline{a}}$ is $\mathrm{vol}_{r-1} (\Delta_{\underline{a}}) = \frac{1}{(r-1)!} \frac{\gcd(a_1, ..., a_r)}{a_1 ... a_r} \mathrm{vol}_{r-1}(C_H)$;
\end{lem}
\begin{proof}
	By Lemma \ref{lemlattice1}, it suffices to show that $\mathrm{vol}_{r-1}(\Delta_{\underline{a}}) = \frac{1}{(r-1)!} \frac{\sqrt{\sum_{1 \leq i \leq r} a_i^2}}{a_1 ... a_r}$. To perform this computation, we can for example use the standard parametrization of $\Delta_{\underline{a}}$ given by $\psi : t \in \Delta   \longmapsto (\frac{1}{a_1} t_1, ..., \frac{1}{a_{r-1}} t_{r-1}, \frac{1}{a_r} (1 - \sum_{1 \leq i \leq r - 1} t_i))$, where $\Delta = \{ (t_i) \in [0,1]^{r-1} \; | \; \sum_i t_i \leq 1 \}$ is the standard $(r-1)$-dimensional simplex in $\mathbb R^{r-1}$. We have then $\psi^\ast( d\mathrm{vol}_{r-1} ) = \sqrt{\det G} \, d\mathrm{vol}_{r-1}$, where $G= (\left< \psi_\ast(e_i), \psi_\ast(e_j)\right>)_{i,j}$ is the Gram matrix of the vectors $\psi_\ast (e_i)$ ($(e_i)_i$ being the canonical basis of $\mathbb R^{r-1}$). A simple computation shows that $\det G = \frac{1}{\prod_i a_i^2} \sum_{i} a_i^2$. Thus, we have $\mathrm{vol}_{r-1}(\Delta_{\underline{a}}) = \frac{\sqrt{\sum_i a_i^2}}{\prod_i a_i} \mathrm{vol}_{r-1}(\Delta)$. To conclude, it suffices to compute $\mathrm{vol}(\Delta) = \frac{1}{(r-1)!}$, which is easy.
\end{proof}

\subsection{Probability estimates on affine simplexes} \label{sectprob}

We present now a few estimates for the classical probability functional on random variables with values in affine simplexes. The following computations are very close to the ones of Demailly in \cite{dem11}, so we tried to give only the necessary details. The main result of this section is Lemma \ref{lemboundvariance}, which was used in the proof of Lemma \ref{lemvariance}.

Again, we use the notations introduced in Definition \ref{defisimplex}. Recall that for any $m$-dimensional simplex $\Delta \accentset{\circ}{\subset} \mathbb R^m$, the \emph{uniform probability measure} of $\Delta$ is the measure $d \mathbf P_{\Delta} = \frac{1}{\mathrm{vol}_m(\Delta)} d \mathrm{vol}_m$. 

Since the uniform measure on $\Delta$ is the unique probability measure which is the restriction of a translation invariant measure on $\mathbb R^m$, we see that if $\Delta_1, \Delta_2 \subseteq \mathbb R^m$ are $m$-dimensional simplexes, and if $\Psi \in \mathrm{GL}(\mathbb R^m)$ is such that $\Delta_2 = \Psi(\Delta_1)$, then $\Psi$ sends the uniform measure of $\Delta_1$ on the uniform measure of $\Delta_2$.
\medskip

Let now $r, k \in \mathbb N$, and consider a random variable $X$ drawn uniformly in the $(kr -1)$-dimensional simplex
$$
\Delta_{\underline{k}} = \Delta_{(1, ..., 1, ..., k, ..., k)} \subseteq \mathbb R^{kr - 1}
$$ 
(each integer $i \in \llbracket 1, k \rrbracket$ being repeated $r$ times). We write $X = (X_{j,l})_{1 \leq j \leq k, 1 \leq l \leq r}$.

\begin{lem} \label{lemrandomvariables} For all $j \in \llbracket 1, k \rrbracket$ and all $i \in \llbracket 1, r \rrbracket$, we let $Y_j = \sum_{1 \leq l \leq r} X_{j, l}$, and $Z_{l}^j = \frac{X_{j,l}}{Y_j}$.
Then
\begin{enumerate}
\item the random variables $Z^j = (Z^j_1, ..., Z^j_r)$ $(1 \leq j \leq k)$ are of uniform law with values in $\Delta^{r-1}$, and are pairwise independent. They are also independent of the $Y_j$ ;
\item the random variable $(Y_1', ..., Y_k') = (Y_1, 2 Y_2, ..., k Y_k)$ takes its values in $\Delta^{k-1} \subseteq \mathbb R^k$. Its density is
$$
dP(y_1, ..., y_{k}) = \frac{(kr - 1)!}{(r - 1)^k} (y_1...y_k)^{r-1}d\mathrm{vol}_{k-1} 
$$
\end{enumerate}
\end{lem}
\begin{rem}
	The density above is a particular case of Dirichlet distribution : it also appears naturally in Demailly's estimates (see \cite[(2.16)]{dem11}).
\end{rem}
\begin{proof}
\emph{(1)} This is easy to check.

\emph{(2)} Let $Y' = (Y_1', ..., Y_k')$, and let, for all $i, j$,  $X_{i,j}' = j X_{i,j}$. Since the uniform measure on a simplex is invariant under linear automorphisms of $\mathbb R^{kr}$, we see that $X' = (X'_{i,j})$ is a random variable with uniform law in the simplex $\Delta^{kr -1}$. We then have $Y_j' = \sum_{1 \leq l \leq r} X'_{l, j}$.

Let $\pi : x \in \Delta^{kr -1} \subseteq \mathbb R_+^{kr}  \longmapsto (\sum_{1 \leq l \leq r} x_{l, 1}, ..., \sum_{1 \leq l \leq r} x_{l, k}) \in \Delta^{k-1}$. We have by definition $Y' = \pi(X')$, so the probability law of $Y$ is the image measure $\pi_{\ast} (d \mathbf{P}_{\Delta^{kr -1}})$. If $y \in \Delta^{k-1}$, we have $\pi_{\ast}^{-1}(y) = \left\{(x_1, x_2, ..., x_k) \in (\mathbb R_+^r)^k \; | \; \forall i, \; \sum_{1 \leq l \leq r} (x_i)_l = y_i \right\}$.
This set is a product of $k$ different $(r-1)$-dimensional simplexes, and we see right away that its euclidian volume is proportional to $y_1^{r-1} ... y_k^{r-1}$. Thus, the probability density of $Y'$ in $\Delta^{k-1}$ is of the form $dP(y) = C y_1^{r-1} ... y_k^{r-1}$. To compute the constant $C$, we apply the normalization $\int_{\Delta^{k-1}} dP(y) = 1$, using inductively the formula $\int_0^1 y^a (1 - y)^b dy = \frac{a! b!}{(a+ b + 1)!}$. A simple computation yields $C = \frac{(kr - 1)!}{(r- 1)!^k}$, as announced. 

\end{proof}

\begin{lem} \label{lemrandomy} With the same notations as in Lemma \ref{lemrandomvariables}, we have
\begin{enumerate}
\item for all $j \in \llbracket 1, k \rrbracket$, $\mathbf{E}[Y_j] = \frac{1}{j k}$ and $\mathbf{E}[Y_j^2] =  \frac{1}{j^2} \frac{r+1}{k(kr + 1)} \leq \frac{2}{j^2 k^2}$. 
\item for all $j, l \in \llbracket 1, k \rrbracket$, with $j \neq l$, we have $\mathbf{E}[Y_j Y_l] \leq \mathbf{E}[Y_j] \mathbf{E}[Y_l]$ (i.e. the variables $Y_j$ and $Y_k$ are negatively correlated). 
\end{enumerate}
\end{lem}
\begin{proof}
For (1), we compute $\mathbf{E}[Y_j] = \frac{1}{j} \mathbf{E}[Y_j'] = \frac{1}{j} \int_{\Delta^{k-1}} y_j dP(y_1, ..., y_k)$. The formula for $dP(y)$ given in Lemma \ref{lemrandomvariables} gives the result. The computation of $\mathbf{E}[Y_j^2]$ is similar. 

For the second point, we write $\mathbf{E}[Y_j Y_l] = \frac{1}{jl} \mathbf{E}[Y_j' Y_l'] = \frac{1}{jl} \int_{\Delta^{k-1}} y_j y_l \, dP(y)$. This gives $\mathbf{E}[Y_j Y_l] = \frac{1}{j l} \frac{r}{k(kr + 1)} \leq \frac{1}{j k} \frac{1}{l k}$, hence the result by (1).
\end{proof}

Now, we let $d_1, ..., d_r \in \mathbb R$, and we let $T$ be a random variable of uniform law on the $(r-1)$-dimensional simplex $\Delta^{r-1} \subseteq \mathbb R^r$. Let $S = \sum_{1 \leq l \leq r} d_l T_l$. We check easily that $\mathbf{E}[S] = \frac{1}{r} \sum_{1 \leq l \leq r} d_l$. 

\begin{lem} \label{lemboundvariance} We let $X = (X_{j,l})_{1 \leq j \leq k, \, 1 \leq l \leq r}$ have the same meaning as before. Let $A : t \in \Delta_{\underline{k}} \longrightarrow  \sum_{1 \leq j \leq  k} \sum_{1 \leq l \leq r} t_{j, l} d_l$. Then, we have an upper bound
\begin{equation} \label{eqvar}
\mathbf{Var}[A(X)] \leq \frac{2}{k^2} \left( \sum_{1 \leq j \leq k} \frac{1}{j^2} \right) \mathbf{E}[S^2] \leq \frac{\pi^2}{3 k^2} \mathbf{E}[S^2].
\end{equation} 
\end{lem}
\begin{proof}
Let $W$ be the left hand side member of \eqref{eqvar}. We have $A(X) = \sum_{j, l} X_{j, l} \, d_l$, so, using the notations of Lemma \ref{lemrandomvariables}, we can rewrite $W$ as follows:
\begin{align*}
W = \mathbf{Var}[\sum_{1 \leq j \leq k} Y_j \left( \sum_{1 \leq i \leq r} d_i Z_{i}^j \right)] = \mathbf{Var}[\sum_{1 \leq j \leq k} Y_j S_j ]
\end{align*}
where we let $S_j = \sum_{1 \leq l \leq r} d_l Z^j_l$. By definition of the variance, we have
\begin{align*}
W & = \mathbf{Var}[\sum_{1 \leq j \leq k} Y_j S_j] = \mathbf{E}[(\sum_{1 \leq j \leq k} Y_j S_j)^2] - \mathbf{E} [\sum_{1 \leq j \leq k} Y_j S_j]^2  
\end{align*}

Now, by Lemma \ref{lemrandomvariables}, the variables $S_j$ are pairwise independent for $j \in \llbracket 1 , k \rrbracket$, and they are also independent of the $Y_j$. Also, by the same lemma, they have the same law than $S$. This implies that for all $j$, we have $\mathbf{E}[Y_j S_j] = \mathbf{E}[Y_j] \mathbf{E}[S_j] = \mathbf{E}[Y_j] \mathbf{E}[S]$, and for all $p \neq q$, $\mathbf{E}[Y_p Y_q S_p S_q] = \mathbf{E}[Y_p Y_q] \mathbf{E}[S_p] \mathbf{E}[S_q] =  \mathbf{E}[Y_p Y_q] \mathbf{E}[S]^2$. Thus, expanding the computation yields 
\begin{align*}
W & = \mathbf{E}[ \sum_j Y_j^2 S_j^2] + \mathbf{E}[\sum_{p \neq q} Y_p Y_q S_p S_q ] -  \sum_j \mathbf{E}[ Y_j S_j]^2 - \sum_{p \neq q} \mathbf{E}[Y_p S_p] \mathbf{E}[Y_q S_q]\\
& = \mathbf{E}[\sum_j Y_j^2]\, \mathbf{E}[S^2] + \sum_{p \neq q} \left( \mathbf{E}[Y_p Y_q] - \mathbf{E}[Y_p] \mathbf{E}[Y_q] \right) \, \mathbf{E}[S^2] - \sum_j \mathbf{E}[Y_j]^2 \, \mathbf{E}[S]^2 
\end{align*}

By Lemma \ref{lemrandomy}, (2), we get
\begin{align*}
W & \leq \sum_j \mathbf{E}[Y_j^2] \mathbf{E}[S^2] - \sum_{j} \mathbf{E}[Y_j]^2 \mathbf{E}[S]^2 \\
 & \leq  \sum_j \mathbf{E}[Y_j^2] \mathbf{E}[S^2],
\end{align*}
and then the conclusion follows immediately from Lemma \ref{lemrandomy}, (1).
\end{proof}

\bibliographystyle{amsalpha}
\bibliography{biblio}

\end{document}